\documentclass[final,onefignum,onetabnum,dvipsnames]{siamart251216}
\usepackage{amsmath, amssymb, mathtools}
\usepackage[absolute,overlay]{textpos}
\usepackage{graphicx}

\usepackage[nameinlink]{cleveref}
\usepackage{algorithm}
\usepackage{algorithmic}
\numberwithin{algorithm}{section}
\usepackage{multirow}
\usepackage{comment}
\usepackage{pifont}

\numberwithin{equation}{section}
\newsiamremark{remark}{Remark}
\newsiamthm{property}{Property}
\crefname{theorem}{Theorem}{Theorems}
\crefname{corollary}{Corollary}{Corollaries}
\crefname{subsection}{section}{sections}
\Crefname{subsection}{Section}{Sections}
\crefname{algorithm}{Algorithm}{Algorithms}
\crefname{table}{Table}{Tables}

\AtBeginDocument{%
  \apptocmd{\appendix}{%
    \crefalias{section}{appendix}%
  }{}{}
}

\crefname{property}{Property}{Properties}
\Crefname{property}{Property}{Properties}
\crefname{figure}{Figure}{Figures}
\crefname{appendix}{Appendix}{Appendices}

\newcommand{\vmark}{\textcolor{Green}{\ding{51}}}%
\newcommand{\xmark}{\textcolor{Red}{\ding{55}}}%
\newcommand{\nmark}{{\textcolor{Orange}{\ \ding{169}}}}
\newcommand{\red}[1]{\textcolor{Red}{#1}}
\newcommand{\orange}[1]{\textcolor{Orange}{#1}}

\newcommand{\new}[1]{#1}

\DeclareMathOperator*{\argmin}{\mathrm{argmin}}

\headers{Halving the size of skew-symmetric eigenvalue problems}{Daniel Kressner and Simon Mataigne}

\title{Halving the size of skew-symmetric eigenvalue problems via the polar decomposition\thanks{Submitted to the editors.\funding{Simon Mataigne is a Research fellow of the Fonds de la Recherche Scientifique - FNRS. This work was supported by the Fonds de la Recherche Scientifique - FNRS under Grant no T.0001.23}}}

\author{Daniel Kressner\thanks{Institute of Mathematics, EPFL, Lausanne, Switzerland (\email{daniel.kressner@epfl.ch})}
\and Simon Mataigne\thanks{ICTEAM Institute, UCLouvain, Louvain-la-Neuve, Belgium (\email{simon.mataigne@uclouvain.be)}}}

\usepackage{amsopn}
\DeclareMathOperator{\diag}{diag}

\ifpdf
\hypersetup{
  pdftitle={Halving the size of skew-symmetric eigenvalue problems},
  pdfauthor={Daniel Kressner and Simon Mataigne}
}
\fi

\begin{document}

\maketitle

\begin{abstract}
This paper introduces a novel algorithm for computing eigenvalues and eigenvectors of a dense real skew-symmetric matrix $A$. Its main ingredient is the computation of a polar factor of $A$ that is both skew-symmetric and orthogonal. This polar factor is then used to transform the original problem into a Hermitian eigenvalue problem of half the size, which can be solved accurately and efficiently with standard software such as LAPACK. Numerical experiments demonstrate the stability of the method and show that its running time is competitive with existing approaches for skew-symmetric eigenvalue problems. Finally, we show that the same principle can be used to reduce 
an orthogonal eigenvalue problem to a unitary eigenvalue problem of half the size.
\end{abstract}
\begin{keyword}
Skew-symmetric matrices, orthogonal matrices, eigenvalue problem, Hamiltonian matrices, symplectic matrices.
\end{keyword}
\begin{MSCcodes}
15A18, 15A20, 65F15, 15B30, 15B57, 15B10.
\end{MSCcodes}

\section{Introduction}\label{sec:introduction}

Computing an eigenvalue decomposition of a real skew-sym\-metric matrix ($A=-A^\top$) is a fundamental problem in linear algebra. While the
skew-symmetric eigenvalue problem (EVP) is less ubiquitous than its symmetric
counterpart, it arises naturally in a number of applications in
science and engineering. For example, skew-symmetric matrices and EVPs play an important role in computational physics~\cite{Penke20}, especially in connection with conservation laws~\cite{DucrosEtAl2000,Gassner2013,SvardNordstrom2014,Tadmor1984}, and in geometric numerical integration~\cite[Chap.~IV]{HairerLubichWanner2006}. Skew-symmetric matrices are also fundamental in Riemannian statistics~\cite{KrakowskiHuperManton2008,Pennec2006} and Riemannian  optimization~\cite{AbsilMahonySepulchre2008}, where they appear prominently in the tangent spaces of the special orthogonal group, the Stiefel manifold, and the Grassmann manifold~\cite{BendokatZimmermannAbsil2024,ZimmermannHuper2022}. Moreover, they find applications in graph theory through the spectral analysis of directed graphs~\cite{CaversEtAl2012,GuoMohar2017,WongMaTian2016}. Finally, as shown in~\cite{mataigne2026jacobi,mataigne2026eigenvalue}, the spectral decomposition of the skew-symmetric part of a real normal matrix can be used to compute its eigenvalue decomposition, in particular when the matrix is orthogonal.
Despite the numerous applications mentioned above, algorithms and implementations for skew-symmetric EVPs are comparably less developed.

The classical approach for skew-symmetric EVPs reduces the matrix by
orthogonal Householder transformations to skew-symmetric tridiagonal
form~\cite{Wilkinson62,MR501120}. The reduced problem can be treated
as a bidiagonal singular value problem~\cite{WardGray78}, or transformed
by a diagonal unitary similarity into a real symmetric tridiagonal
EVP, which partly ignores the skew-symmetric structure. The latter approach was implemented on distributed-memory systems
by Shao et al.~\cite{MR3419779}. More recently, Penke
et al.~\cite{Penke20} developed such a solver for ELPA~\cite{ELPA}, including
an adaptation of successive band reduction~\cite{MR1941085}
for performing skew-symmetric tridiagonalization. In passing, we note that 
Jacobi-like algorithms for skew-symmetric have also been proposed and analyzed in~\cite{MR1227767,Paardekooper1971}.

Dedicated, publicly available software for real skew-symmetric eigenvalue problems remains uncommon. In particular, widely used general-purpose libraries such as LAPACK~\cite{lapack99}, as well as the major GPU-oriented libraries cuSOLVER~\cite{NVIDIAcuSOLVER} and MAGMA~\cite{Tomov2010MAGMA} provide functionality for general and symmetric/Hermitian EVPs, but no dedicated skew-symmetric eigenvalue solver. Notable occurrences in specialized libraries include the ELPA solver from~\cite{Penke20}, the ScaLAPACK-style solver in BSEPACK from~\cite{MR3419779}, as well as the Julia package SkewLinearAlgebra.jl. 

Because of the limited availability of specialized software, a common workaround is to exploit the skew-symmetric structure only partially and apply, for example, a Hermitian eigensolver to $\mathrm{i} A$. Since this converts a real EVP into a complex EVP of the \emph{same} size, it can be considerably more expensive.

\paragraph{Contributions}
This work proposes a new approach that fully leverages the structure of a skew-symmetric EVP by reducing it to a Hermitian EVP of \emph{half} the size. For this purpose, we compute a skew-symmetric and orthogonal polar factor of $A$ using the matrix iterations from~\cite{NakatsukasaBaiGygi2010,NakatsukasaHigham2012}, whose main
computational kernels---QR factorizations, Cholesky factorizations, and matrix multiplications---are well suited to parallel architectures~\cite{Ltaief19}. The polar factor then yields a compression of the original problem to a half-size Hermitian EVP, which can be solved using standard Hermitian eigensolvers from general-purpose libraries. We demonstrate the numerical stability of the algorithm with experiments and we show that it can be competitive in terms of execution time with respect to alternative approaches. \new{We develop a similar method to reduce the special orthogonal EVP to a half-size unitary EVP.}

\paragraph{Reproducibility} An implementation of our method in Julia~\cite{Julia-2017}, as well as the codes to perform all experiments presented in this paper, are available at \url{https://github.com/smataigne/SkewSchur.jl}.

\paragraph{Organization of the paper}
\Cref{sec:preliminaries} recalls relevant properties of skew-symmetric matrices and polar decompositions.
Then, \Cref{sec:hamiltonian_reduction} develops the transformation of a
skew-symmetric eigenvalue problem to a half-size Hermitian eigenvalue problem.
The resulting algorithm is presented in~\Cref{sec:algorithm}, including implementation details and the treatment of matrices of odd size.
Numerical experiments regarding the accuracy and the running time of the algorithm are conducted in~\Cref{sec:experiments}, which also illustrates its use for the simulation of a hyperbolic partial differential equation.
Finally, \Cref{sec:special_orthogonal} explains how the same ideas can be adapted to reduce special
orthogonal eigenvalue problems to unitary eigenvalue problems of half the size. 

\paragraph{Notation} $\mathrm{Skew}(n)$ and $\mathrm{Sym}(n)$ denote, respectively, the sets of $n\times n$ skew-symmetric and symmetric matrices, while $\mathrm{skew}:\mathbb{R}^{n\times n}\rightarrow \mathrm{Skew}(n)$ and $\mathrm{sym}:\mathbb{R}^{n\times n}\rightarrow \mathrm{Sym}(n)$ return the skew-symmetric and symmetric parts of a matrix.
The orthogonal and unitary groups of $n\times n$ matrices are denoted by $\mathrm{O}(n)$ and $\mathrm{U}(n)$, respectively. 
The $n\times n$ identity matrix is denoted by $I_n$ and we set
\begin{equation*}
    J_{2n} \coloneq \begin{bmatrix}
        0&-I_n\\
        I_n&0
    \end{bmatrix} = \begin{bmatrix}
        0&-1\\
        1&0
    \end{bmatrix}\otimes I_n,
\end{equation*}
where  $\otimes$ denotes the Kronecker product. We use $\oplus$ to denote the direct sum of matrices.

\section{Preliminaries on skew-symmetric matrices}\label{sec:preliminaries}

This section recalls relevant properties of skew-symmetric eigenvalue problems and establishes several results concerning Hamiltonian matrices and polar decompositions.

\subsection{Spectral decomposition of skew-symmetric matrices}

We start by recalling the structure of a \emph{real spectral decomposition} (RSD) of a skew-symmetric matrix. 
\begin{lemma}[{\cite[Cor.~2.5.11]{Horn_Johnson_2012}}]\label{lem:rsd_existence}
    For every $A\in\mathrm{Skew}(2n)$, there exist an orthogonal matrix $Q\in\mathrm{O}(2n)$ and a diagonal matrix $\Sigma\in\mathbb{R}^{n\times n}$ with nonnegative diagonal entries such that
    \begin{equation}\label{eq:spectral_decomposition}
        A = Q \begin{bmatrix}
                0&-\Sigma\\
                \Sigma& 0
            \end{bmatrix}Q^\top = Q(J_2 \otimes \Sigma)Q^\top \eqcolon QSQ^\top.
    \end{equation}
    The factorization $A = QSQ^\top$ is called a \emph{real spectral decomposition (RSD)} of $A$. In slight abuse of notation, the matrix $Q$ is termed a \emph{spectral basis}.
\end{lemma}

If $A\in \mathrm{Skew}(2n+1)$, the matrix is singular and $\dim(\ker(A))$ is odd. An RSD takes the form $A  = Q((J_2 \otimes \Sigma) \oplus 0_{1\times 1})Q^\top $ with $Q \in \mathrm{O}(2n+1)$.  In the following, we will focus on even-sized matrices and return to the odd case in~\Cref{sec:odd}.

The RSD of a skew-symmetric matrix is never unique. For example, for every diagonal matrix $\Phi\in\mathbb{R}^{n\times n}$, we have
\begin{equation*}
    \begin{bmatrix}
        \cos(\Phi)&-\sin(\Phi)\\
        \sin(\Phi)&\cos(\Phi)
    \end{bmatrix}\begin{bmatrix}
                0&-\Sigma\\
                \Sigma& 0
            \end{bmatrix}\begin{bmatrix}
        \cos(\Phi)&-\sin(\Phi)\\
        \sin(\Phi)&\cos(\Phi)
    \end{bmatrix}^\top = \begin{bmatrix}
                0&-\Sigma\\
                \Sigma& 0
            \end{bmatrix}.
\end{equation*}

An important step in our approach is to transform a skew-symmetric matrix orthogonally into skew-symmetric Hamiltonian form.

\begin{definition}\label{def:hamiltonian}
    A skew-symmetric Hamiltonian matrix $A_\mathcal{H}\in\mathbb{R}^{2n\times 2n}$ takes the form
    \begin{equation} \label{eq:sham}
        A_\mathcal{H} = \begin{bmatrix}
            \Omega&-H\\
            H&\Omega
        \end{bmatrix} = I_2 \otimes \Omega + J_2 \otimes H,
    \end{equation}
    with $\Omega \in\mathrm{Skew}(n)$ and $H\in\mathrm{Sym}(n)$.
    The set of $2n \times 2n$ skew-symmetric Hamiltonian matrices is denoted by $\mathrm{SkH}(2n)$.
\end{definition}

An important set related to Hamiltonian matrices is the group $\mathrm{OSp}(2n)$ of orthogonal symplectic matrices, i.e., the set of $2n\times 2n$ orthogonal matrices commuting with $J_{2n}$~\cite{Dopico09,Fassbender01}:
\begin{equation*}
    \mathrm{OSp}(2n)\coloneq \{M\in \mathrm{O}(2n)\ | \ MJ_{2n} = J_{2n} M\}.
\end{equation*}
The commutativity with $J_{2n}$ induces a specific block structure:
\begin{equation}\label{eq:symplectic_blocks}
    M\in\mathrm{O}(2n)\ \text{and} \ MJ_{2n} = J_{2n} M \iff M = \begin{bmatrix}
            U_\mathrm{i}&-U_\mathrm{r}\\
            U_\mathrm{r}&U_\mathrm{i}
        \end{bmatrix}, \ U_\mathrm{r} + \mathrm{i} U_\mathrm{i} \in \mathrm{U}(n).
\end{equation}
This yields the fundamental result that orthogonal symplectic matrices are isomorphic to unitary matrices: $ \mathrm{OSp}(2n)\simeq \mathrm{U}(n)$.

Likewise, skew-symmetric Hamiltonian matrices of size $2n$ are in one-to-one correspondence with Hermitian matrices of size $n$. As a consequence, an eigenvalue decomposition
\[
 H+\mathrm{i} \Omega = (U_\mathrm{r} + \mathrm{i} U_\mathrm{i})\widetilde{\Sigma} (U_\mathrm{r} + \mathrm{i} U_\mathrm{i})^*
\]
yields the block diagonalization
\begin{equation}\label{eq:hermitian_evp}
        \begin{bmatrix}
            \Omega&-H\\
            H&\Omega
        \end{bmatrix}=\begin{bmatrix}
            U_\mathrm{i}&-U_\mathrm{r}\\
            U_\mathrm{r}&U_\mathrm{i}
        \end{bmatrix} \begin{bmatrix}
            0&-\widetilde{\Sigma}\\
            \widetilde{\Sigma} & 0
        \end{bmatrix} \begin{bmatrix}
            U_\mathrm{i}&-U_\mathrm{r}\\
            U_\mathrm{r}&U_\mathrm{i}
        \end{bmatrix}^\top,
    \end{equation}
    where $\widetilde{\Sigma} \in \mathbb R^{n\times n}$ is diagonal (but its diagonal entries are not necessarily nonnegative); see also~\cite{Fassbender01,Tisseur2001}. An RSD is obtained from~\eqref{eq:hermitian_evp} by replacing 
$\widetilde{\Sigma}$ with $|\widetilde{\Sigma}|$
and applying suitable sign changes to the corresponding basis vectors.

Given $A \in\mathrm{Skew}(2n)$, equation~\eqref{eq:hermitian_evp} suggests a natural route for solving skew-symmetric eigenvalue problems: First, one computes $Z \in \mathrm{O}(2n)$ such that $A_\mathcal{H} = Z^\top A Z$ is skew-symmetric Hamiltonian. Then, the associated Hermitian eigenvalue problem of half the size is solved using standard libraries such as LAPACK~\cite{lapack99}. 


\subsection{Polar decomposition of skew-symmetric matrices}

The polar decomposition plays a central role in constructing the matrix $Z$ needed in the approach outlined above.
Let us recall the definition of a polar decomposition.
\begin{definition}[{\cite[Thm.~9.4.1]{golubvanloan}}] \label{def:polar}
    For every matrix $A\in\mathbb{R}^{n\times n}$, a polar decomposition is $A = PY$ where $P\in\mathrm{O}(n)$ and $Y\in\mathrm{Sym}(n)$ is positive semidefinite ($Y\succeq 0$). The matrix $P$ is called an \emph{orthogonal polar factor} of $A$.
\end{definition}

    

For matrices of even size, our method requires the selection of an orthogonal polar factor $P$ that is also \emph{skew-symmetric}.
The existence of such a skew-symmetric orthogonal polar factor (\emph{skopf}) follows from~\cite[Thm~5.4]{MR2208338}. The following lemma makes the additional observation that any such polar factor and the original matrix admit a common spectral basis.
\begin{lemma}\label{lem:polar_decomposition}
     Let $A\in\mathrm{Skew}(2n)$. Then $P$ is a skew-symmetric orthogonal polar factor (\emph{skopf}) of $A$ if and only if there exist $Q\in\mathrm{O}(2n)$ and a nonnegative diagonal matrix $\Sigma \in \mathbb R^{n \times n}$ such that $P = QJ_{2n}Q^\top$ and $A = Q(J_2 \otimes \Sigma)Q^\top$.
\end{lemma}
\begin{proof}
    ($\Longleftarrow $) Suppose that $P = QJ_{2n}Q^\top$ and $A=Q(J_2 \otimes \Sigma)Q^\top$ with $Q,\Sigma$ as
    in the statement of the lemma. Using the identity $ J_{2}\otimes \Sigma = J_{2n} (I_2 \otimes \Sigma)$, we directly obtain 
    \begin{equation}\label{eq:polar_decomposition}
        A =  Q (J_2 \otimes \Sigma)Q^\top \\
            =  \left( Q J_{2n} Q^\top\right)\left( Q (I_2\otimes \Sigma)Q^\top\right).
    \end{equation}
    Because of $Q( I_2\otimes \Sigma) Q^\top\succeq 0$, this is a polar decomposition with the skopf $P$ of $A$.
   
    ($\Longrightarrow$) Assume that $A = PY$ is a polar decomposition with $P\in\mathrm{O}(2n)\cap \mathrm{Skew}(2n)$. 
    
    If $A$ is invertible then $P$ is unique~\cite[Chap.~9.4.3]{golubvanloan}. Therefore, considering an RSD $A=Q(J_2 \otimes \Sigma)Q^\top$, the identity $P=QJ_{2n} Q^\top$ holds by uniqueness and \eqref{eq:polar_decomposition}.

    If $A$ is singular then it has rank $2r< 2n$. Then there is $V\in\mathrm{O}(2n)$ such that $V^\top A V = \widetilde{A} \oplus 0_{2(n-r)}$ where $\widetilde{A}\in\mathrm{Skew}(2r)$ is invertible. Consider a skopf $P$ of~$A$. Then $\widehat{P}:= V^\top P V$ is an orthogonal polar factor of
    $\widetilde{A} \oplus 0_{2(n-r)}$, which, by \cref{lem:polar_plus}, necessarily takes the form
    \[\widehat{P} = \widetilde{Q}J_{2r}\widetilde{Q}^\top \oplus R.\]
    Here, by the discussion above, the first diagonal block
    $\widetilde{Q}J_{2r}\widetilde{Q}^\top$ is the unique polar factor of $\widetilde{A} = \widetilde{Q} (J_2 \otimes \widetilde{\Sigma})\widetilde{Q}^\top$. For the second diagonal block, an arbitrary $R\in\mathrm{O}(2(n-r))$ is an orthogonal polar factor of $0_{2(n-r)}$. One observes that $\widehat{P}$, and thus $P$, are skew-symmetric if and only if $R$ is skew-symmetric. This implies that $R = \widetilde{Q}_RJ_{2(n-r)}\widetilde{Q}_R^\top$ for some $\widetilde{Q}_R\in\mathrm{O}(2(n-r))$. Note that
    \begin{equation*}
        A = V \begin{bmatrix}
            \widetilde{Q}&0\\
            0&\widetilde{Q}_R
        \end{bmatrix} \begin{bmatrix}
            (J_2\otimes \widetilde{\Sigma})&0\\
            0&0_{2(n-r)}
        \end{bmatrix}  \begin{bmatrix}
            \widetilde{Q}^\top&0\\
            0&\widetilde{Q}_R^\top 
        \end{bmatrix} V^\top  ,  
    \end{equation*}
    is a permuted RSD. 
    Letting $\Sigma \coloneq \widetilde{\Sigma} \oplus 0_{n-r}$ and $Q$ be obtained by an appropriate permutation of the columns of $V
    (\widetilde{Q}\oplus \widetilde{Q}_R)$ establishes the decompositions $P=QJ_{2n}Q^\top$ and $A = Q(J_2 \otimes \Sigma)Q^\top$.
\end{proof}

\section{Transformation to skew-symmetric Hamiltonian form}\label{sec:hamiltonian_reduction}

Consider $A\in\mathrm{Skew}(2n)$, and let $P$ be a skopf of $A$. By \cref{lem:polar_decomposition}, there exist $Q\in\mathrm{O}(2n)$ and a nonnegative
diagonal matrix $\Sigma\in\mathbb{R}^{n\times n}$ such that
\[
    A = Q(J_2\otimes\Sigma)Q^\top,
    \qquad
    P = QJ_{2n}Q^\top.
\]
We claim that, in fact, \emph{any} spectral basis of $P$ transforms $A$ to
skew-symmetric Hamiltonian form. Indeed, suppose that $Z\in\mathrm{O}(2n)$
satisfies
\[
    P=ZJ_{2n}Z^\top,
\]
and set $M:=Q^\top Z\in\mathrm{O}(2n)$. Comparing the two spectral
decompositions of $P$ gives
\[
    J_{2n}=MJ_{2n}M^\top.
\]
Thus, $M$ is orthogonal symplectic. Since $J_2\otimes\Sigma$ commutes with $J_{2n}$,
it follows that
\[
    M^\top (J_2\otimes\Sigma) M J_{2n}
    =J_{2n} M^\top (J_2\otimes\Sigma) M.
\]
In other words,
$Z^\top A Z=M^\top (J_2\otimes\Sigma) M$
is skew-symmetric and Hamiltonian.

It remains to construct a spectral basis $Z$ of $P$ in an efficient manner, without
resorting to a general-purpose eigensolver. \Cref{thm:qfactor} below provides such a
construction by extracting the eigenspace of $P$ associated with the
eigenvalue $\mathrm{i}$ from the range of $P+\mathrm{i}I_{2n}$.
Here and in the following, the function $\texttt{qf}(\cdot)$ returns an orthonormal basis of the space spanned by the columns of the input matrix. In practice, it is obtained as the Q-factor of a thin QR factorization.

\begin{theorem} \label{thm:qfactor}
Let $P\in\mathrm{Skew}(2n)\cap\mathrm{O}(2n)$ and consider any
$V\in\mathbb{R}^{2n\times n}$ such that
$(P+\mathrm{i}I_{2n})V$ has full column rank. Define
\[
    \widetilde V:=\text{\emph{\texttt{qf}}}\bigl((P+\mathrm{i}I_{2n})V\bigr), \quad \text{and} \quad
    Z:=\sqrt{2}
    \begin{bmatrix}
        \Re(\widetilde V) & -\Im(\widetilde V)
    \end{bmatrix},
\]
where $\Re(\cdot)$ and $\Im(\cdot)$ denote the real and imaginary parts, respectively.
Then $Z\in\mathrm{O}(2n)$ and
\[
    P=ZJ_{2n}Z^\top.
\]
\end{theorem}

\begin{proof}
Since $P$ is skew-symmetric and orthogonal, 
$P^2=-I_{2n}$ and, hence, 
\[  
    P(P+\mathrm{i}I_{2n})V
    =(P^2+\mathrm{i}P)V
    =(-I_{2n}+\mathrm{i}P)V
    =\mathrm{i}(P+\mathrm{i}I_{2n})V.
\]
Consequently,
$
    P\widetilde V=\mathrm{i}\widetilde V.
$
Writing $\widetilde V=\widetilde{V}_\mathrm{r}+\mathrm{i}\widetilde{V}_\mathrm{i}$ with $\widetilde{V}_\mathrm{r} = \Re(\widetilde V)$ and 
$\widetilde{V}_\mathrm{i} = \Im(\widetilde V)$, and comparing real and imaginary
parts gives
\begin{equation} \label{eq:pvr}
     P\widetilde{V}_\mathrm{r} =-\widetilde{V}_\mathrm{i},
    \qquad
    P\widetilde{V}_\mathrm{i} = \widetilde{V}_\mathrm{r}.
\end{equation}
On the other hand, because $\widetilde V$ has orthonormal columns,
\[
    \widetilde{V}_\mathrm{r}^\top \widetilde{V}_\mathrm{r}+\widetilde{V}_\mathrm{i}^\top \widetilde{V}_\mathrm{i}=I_n,
    \qquad
    \widetilde{V}_\mathrm{r}^\top \widetilde{V}_\mathrm{i} = \widetilde{V}_\mathrm{i}^\top \widetilde{V}_\mathrm{r}.
\]
Combined with~\eqref{eq:pvr} and the orthogonality of $P$, this gives
\[
    \widetilde{V}_\mathrm{i}^\top \widetilde{V}_\mathrm{i}=\widetilde{V}_\mathrm{r}^\top P^\top P \widetilde{V}_\mathrm{r}=\widetilde{V}_\mathrm{r}^\top \widetilde{V}_\mathrm{r} \quad \Longrightarrow \quad \widetilde{V}_\mathrm{r}^\top \widetilde{V}_\mathrm{r} = \widetilde{V}_\mathrm{i}^\top \widetilde{V}_\mathrm{i} = \frac12 I_n.
\]
Also, $\widetilde{V}_\mathrm{r}^\top \widetilde{V}_\mathrm{i}=-\widetilde{V}_\mathrm{r}^\top P \widetilde{V}_\mathrm{r}$ is not only symmetric but also skew-symmetric (because $P$ is skew-symmetric). This implies 
$\widetilde{V}_\mathrm{r}^\top \widetilde{V}_\mathrm{i}=0$. In summary, we have shown that $Z\in\mathrm{O}(2n)$.
Combined with
\[
    PZ
    =\sqrt{2}\begin{bmatrix}P \widetilde{V}_\mathrm{r}&-P \widetilde{V}_\mathrm{i}\end{bmatrix}
    =\sqrt{2}\begin{bmatrix}-\widetilde{V}_\mathrm{i}&-\widetilde{V}_\mathrm{r} \end{bmatrix}
    =ZJ_{2n},
\]
this completes the proof.
\end{proof}

Together with the discussion in the beginning of this section, \cref{thm:qfactor} establishes the following result.
\begin{corollary} \label{thm:polar_rsd}
Let $A\in\mathrm{Skew}(2n)$ and consider the matrix $Z$ constructed in \cref{thm:qfactor} for a skopf $P$ of $A$. Then 
$A_\mathcal{H} = Z^\top A Z$ is a skew-symmetric Hamiltonian matrix.
\end{corollary}

\section{Algorithm and implementation details} \label{sec:algorithm}

The results of \Cref{sec:hamiltonian_reduction} provide the theoretical basis of~\cref{alg:hamiltospectral} for computing an RSD of a skew-symmetric matrix. At step~$1$, a skopf $P$ of $A$ is computed. Steps~$2$ to~$4$
realize the construction of \cref{thm:qfactor} to compute an RSD $P = ZJ_{2n}Z^\top$ with $Z = [Z_1\ |\ Z_2]$.
By \cref{thm:polar_rsd}, the transformed matrix
$A_{\mathcal H}=Z^\top AZ$ is skew-symmetric Hamiltonian.
Step~$5$ computes the blocks 
$\Omega \coloneq Z_1^\top (AZ_1)$ and $H \coloneq Z_2^\top (AZ_1)$ in the partition~\cref{eq:sham}
of the skew-symmetric Hamiltonian matrix $A_\mathcal{H} = Z^\top AZ$. At Step~$6$, the equivalent Hermitian EVP  $H+i\Omega = U\widetilde{\Sigma} U^*$ is solved, which gives the block diagonalization~\eqref{eq:hermitian_evp} of $A$. 
In Step~$7$, an RSD of $A$ is obtained by removing the signs from $\widetilde{\Sigma}$:
\begin{equation*}
    A = Z  \begin{bmatrix}
            U_\mathrm{i}&-U_\mathrm{r}\\
            U_\mathrm{r}&U_\mathrm{i}
        \end{bmatrix} \begin{bmatrix}
            I_n&0\\
            0&\mathrm{sign}(\widetilde{\Sigma} )
        \end{bmatrix}\begin{bmatrix}
            0&-|\widetilde{\Sigma} |\\
            |\widetilde{\Sigma} | & 0
        \end{bmatrix}\begin{bmatrix}
            I_n&0\\
            0&\mathrm{sign}(\widetilde{\Sigma} )
        \end{bmatrix} \begin{bmatrix}
            U_\mathrm{i}&-U_\mathrm{r}\\
            U_\mathrm{r}&U_\mathrm{i}
        \end{bmatrix}^\top Z^\top,
\end{equation*}
where we use the convention $\mathrm{sign}(0)\coloneq 1$. 
In the rest of this section, we discuss the implementation of each step of \cref{alg:hamiltospectral}.
\begin{algorithm}
    \caption{Real spectral decomposition (RSD) of skew-symmetric matrix}
    \begin{algorithmic}
        \STATE \textbf{Input:} $A\in\mathrm{Skew}(2n)$.
        \STATE \textbf{Output:} $Q\in\mathrm{O}(2n)$ and nonnegative diagonal matrix $\Sigma\in\mathbb{R}^{n\times n}$ such that $A = Q (J_2 \otimes \Sigma )Q^\top$.
        \STATE \textbf{step 1:} Compute a skew-symmetric orthogonal polar factor $P$ of $A$.
        \STATE \textbf{step 2:} Sample $V \in \mathbb R^{2n\times n}$ with independent standard normal entries.
        \STATE \textbf{step 3:} Compute an orthonormal basis $\widetilde{V} = \texttt{qf}((P + \mathrm{i}I_{2n})V)\in\mathbb{C}^{2n\times n}$.
        \STATE \textbf{step 4:} Set $Z_1 \coloneq \sqrt{2} \Re(\widetilde{V})$ and $Z_2 \coloneq -\sqrt{2} \Im(\widetilde{V})$.
        \STATE \textbf{step 5:} Compute $\Omega \coloneq  Z_1^\top (A Z_1)$ and $H\coloneq Z_2^\top (A Z_1) $.
        \STATE \textbf{step 6:} Solve the Hermitian EVP of $H + i\Omega = U\widetilde{\Sigma} U^*$ where $U = U_\mathrm{r} + \mathrm{i} U_\mathrm{i}$.
        \STATE \textbf{step 7:} Compute $Q \coloneq  [Z_1U_\mathrm{i} + Z_2U_\mathrm{r}\ |\ (-Z_1 U_\mathrm{r} + Z_2 U_\mathrm{i})\mathrm{sign}(\widetilde{\Sigma} )]$ and $\Sigma\coloneq |\widetilde{\Sigma} |$.
        \RETURN $Q$ and $\Sigma$.
    \end{algorithmic}
    \label{alg:hamiltospectral}
\end{algorithm}

\begin{remark}
   In exact arithmetic, \cref{alg:hamiltospectral}
   also applies to a singular matrix $A$, provided that a skopf is available. In the presence of roundoff error, the case of a (nearly) singular $A$ can make the accurate computation challenging, effectively restricting the set of applicable methods for computing a skopf, as discussed in \Cref{sec:polar_factor}. In practice, the case of an exactly singular matrix can be mitigated by adding a skew-symmetric random matrix $E$ to $A$ of small norm $\|E\|_\mathrm{F}  = u \|A\|_\mathrm{F}$, where $u$ denotes unit roundoff. A decomposition of $A+E$ can still be interpreted in terms of structured backward
error.
\end{remark}

\subsection{Step 1: Computing skopf}\label{sec:polar_factor}

The stable and fast computation of a skopf is critical to the well-functioning of \cref{alg:hamiltospectral}. In \Cref{sec:app_polar_factor} several methods from the literature for computing polar decompositions are adapted to preserve skew-symmetry and compared numerically. Their behavior differs significantly for nearly singular matrices. Our experiments indicate that QDWH~\cite{NakatsukasaBaiGygi2010}, combined with a skew-symmetric projection after each iteration, provides the most robust overall choice. In practice, it produces
small orthogonality and polar-decomposition residuals in our tests,
including for matrices with condition numbers as large as $10^{16}$.

\subsection{Steps~2 to~4: Choosing a sampling matrix $V$ and computing $\widetilde{V}$}
The method requires choosing a matrix $V\in\mathbb{R}^{2n\times n}$ such that the matrix $(P+\mathrm{i} I_{2n})V\in\mathbb{C}^{2n\times n}$ has full column rank and is, preferably, well-conditioned. Note that the eigenvalues of $P+\mathrm{i}I_{2n}$ are $0$ and $2 \mathrm{i}$, each with multiplicity $n$. In particular, $P+\mathrm{i} I_{2n}$ has only~$n$ linearly independent columns. Indeed, using an RSD $P = QJ_{2n}Q^\top$, it follows that
\begin{equation*}
    P + \mathrm{i} I_{2n} = Q\begin{bmatrix}
        \mathrm{i} I_n&-I_n\\
        I_n& \mathrm{i} I_n
    \end{bmatrix}Q^\top.
\end{equation*}
This decomposition reveals the linear dependency and $V = Q\left[\begin{smallmatrix}
        I_n\\
        0
    \end{smallmatrix}\right]$ would be an excellent choice. However, this choice cannot be achieved because $Q$ is unknown. 
    We therefore choose $V$ randomly. More precisely, let
$V\in\mathbb R^{2n\times n}$ be a Gaussian random matrix (that is, $V$ has independent standard normal entries).
Partition
$Q^\top V
    =
    \begin{bmatrix}
        G_1\\G_2
    \end{bmatrix}$ with 
    $G_1,G_2\in\mathbb R^{n\times n}$.
By rotational invariance, $G_1$ and $G_2$ are independent Gaussian random matrices. Setting $G=G_1+\mathrm{i}G_2$, we obtain
\[
    (P+\mathrm{i}I_{2n})V
    =
    Q
    \begin{bmatrix}
        \mathrm{i}I_n&-I_n\\
        I_n&\mathrm{i}I_n
    \end{bmatrix}
    \begin{bmatrix}
        G_1\\G_2
    \end{bmatrix}
    =
    Q
    \begin{bmatrix}
        \mathrm{i}G\\G
    \end{bmatrix}.
\]
Consequently,
$
    \kappa_2\bigl((P+\mathrm{i}I_{2n})V\bigr)
    =
    \kappa_2(G),
$
and $(P+\mathrm{i}I_{2n})V$ has full column rank with probability one.
The condition number of a square complex Gaussian matrix typically
grows linearly with $n$; see
\cite[Thm.~6.2]{Edelman1988RandomEigenvalues}.

    Although $\kappa_2((P+\mathrm{i}I_{2n})V)$ grows linearly with $n$, it
remains moderate for the matrix sizes typically considered.
There is therefore some flexibility in the choice of the
orthogonalization method in Step~3. In our experiments, we use a
blocked Householder QR factorization, which exploits Level~3
BLAS and combines high performance with excellent numerical
stability~\cite[Thm.~19.4]{Higham2002}. More parallel and
communication-efficient alternatives include block Gram--Schmidt
methods with reorthogonalization and CholeskyQR2~\cite{MR3359252}.

    The right plot of \Cref{fig:conditioning} shows how the residual
$\|PZ-ZJ_{2n}\|_{\mathrm F}/\|P\|_\mathrm{F}$ grows with $n$. The observed growth is
consistent with the increasing condition number of
$(P+\mathrm{i}I_{2n})V$ and with the fact that the mutual
orthogonality relations between the real and imaginary parts of
$\widetilde V$ are inherited from the structure rather than
enforced numerically .

    \begin{figure}
        \centering
        \includegraphics[width=0.48\linewidth]{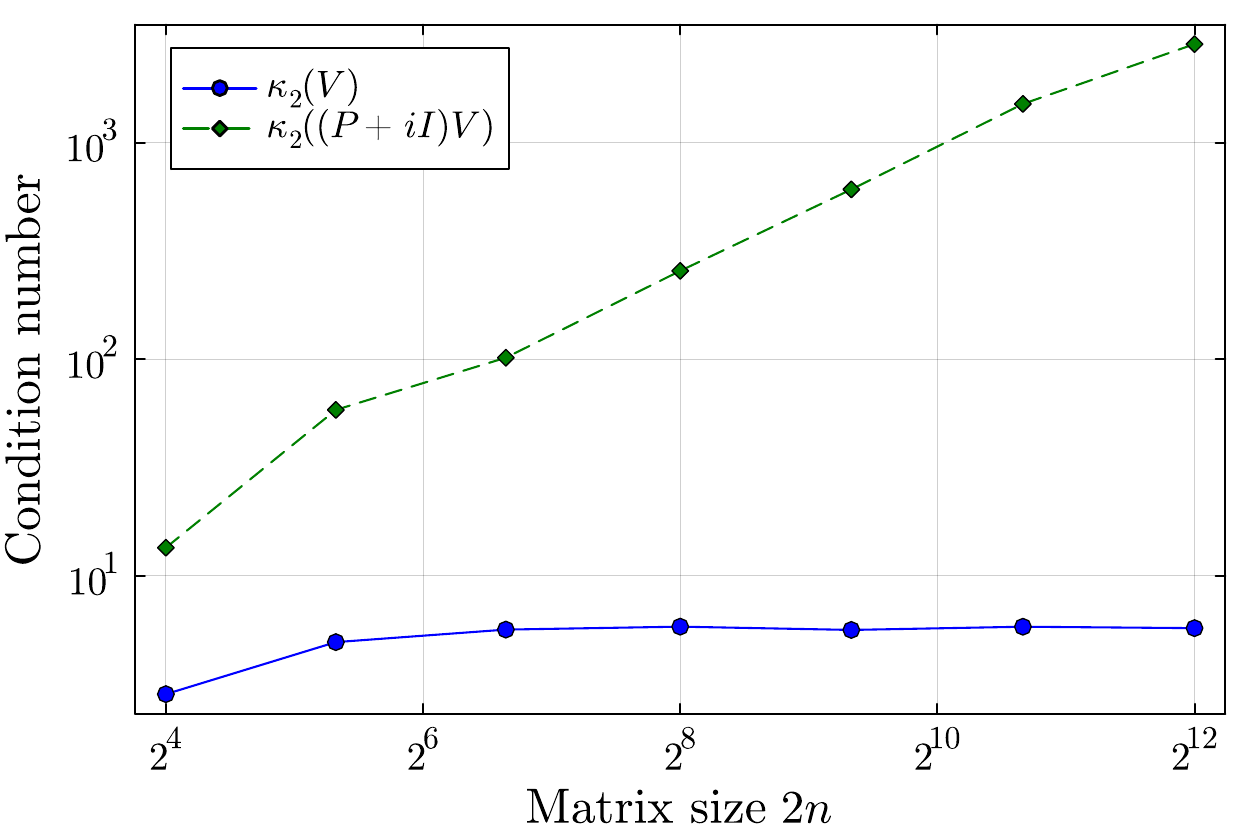}
        \includegraphics[width=0.48\linewidth]{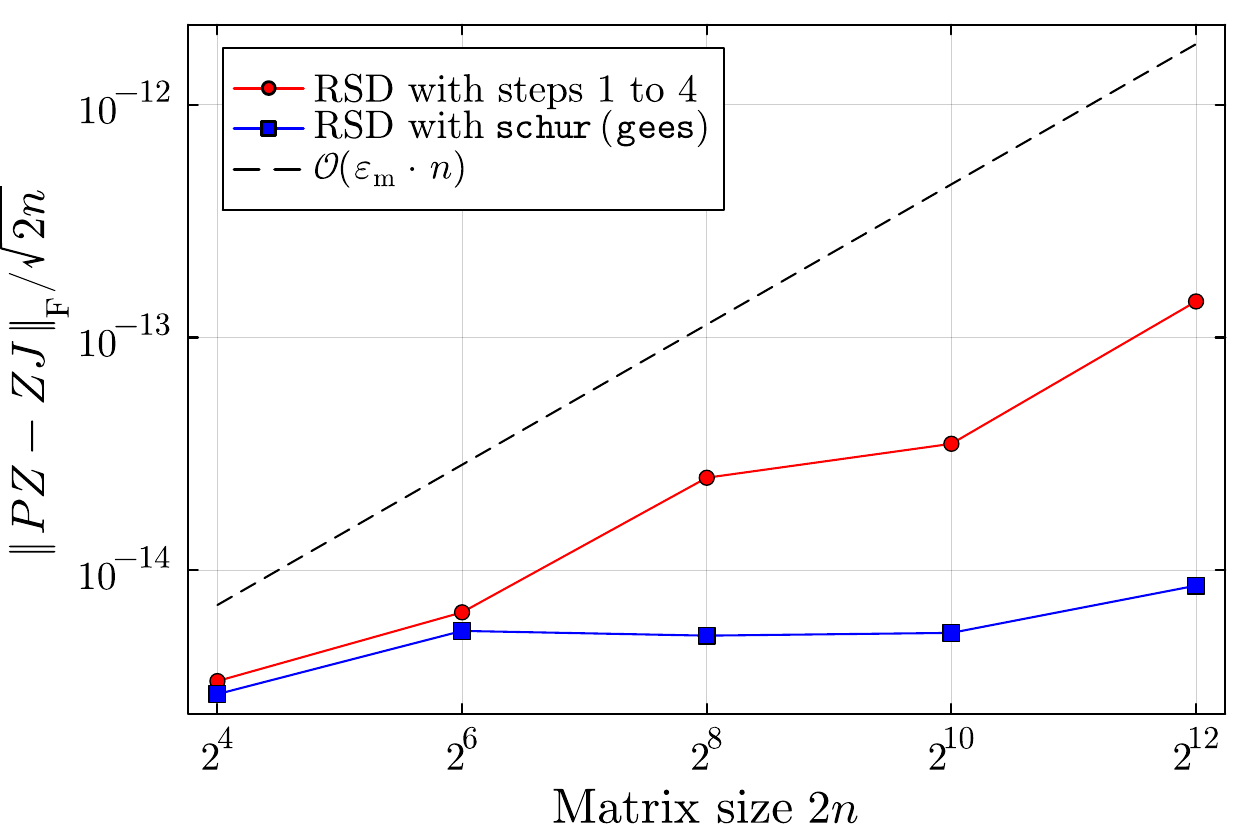}
        \caption{(Left) Evolution of $\kappa_2(V)$ and $\kappa_2((P+\mathrm{i}I_{2n})V)$ where $V = \texttt{randn}(2n,n)$ as $n$ grows. The results are averaged over 100 runs. $\kappa_2(V)$ converges to $\frac{2+\sqrt{2}}{2-\sqrt{2}}$~\cite[Thm.~6.3]{Edelman1988RandomEigenvalues} and $\kappa_2((P+\mathrm{i}I_{2n})V)$ grows linearly with $n$~\cite[Thm.~6.2]{Edelman1988RandomEigenvalues}. (Right) Evolution of the accuracy of the real spectral decomposition of $P$ obtained with steps~1 to 4. It is compared with the routine \texttt{gees} from LAPACK. The results are averaged over 10 runs.}
        \label{fig:conditioning}
    \end{figure}
    
\subsection{Steps~5 to 7: Hamiltonian form and Hermitian EVP} Step~5 consists of matrix multiplications with orthonormal matrices. This operation is numerically stable \cite[Chap.~3.6]{Higham2002}.
To assess the loss of Hamiltonian structure, \cref{tab:hamiltonian_form} measures the numerical stability of \cref{alg:hamiltospectral} up to step 5. As $n$ grows, it shows the evolution of the distance of the numerically computed matrix $A_\mathcal{H} = Z^\top A Z$ to the nearest skew-symmetric Hamiltonian matrix, denoted by $\mathrm{skh}(Z^\top A Z)$; see~\cref{thm:nearest_skh} below.
One observes that the residual grows (mildly) with $n$ but shows little sensitivity
to the conditioning of $A$.
\begin{table}[ht]
    \centering
    \begin{tabular}{||c|c|c|c||}
    \hline
         &\multicolumn{3}{|c||}{$\|A_\mathcal{H}-\mathrm{skh}(A_\mathcal{H})\|_\mathrm{F}/\|A\|_\mathrm{F}$} \\
         \hline
         $2n$&Random matrix&Ill-conditioned matrix&Well-conditioned matrix\\ 
         \hline
         10&8.24e-16&5.92e-16&6.24e-16\\
        100& 5.48e-15& 4.42e-15& 4.61e-15\\
        1000& 2.32e-14& 2.21e-14& 2.95e-14\\
        \hline
    \end{tabular}
    \caption{Numerical experiment on the accuracy of the skew-symmetric Hamiltonian form after step 5 of \cref{alg:hamiltospectral}. The Frobenius distance of $A_\mathcal{H}=Z^\top A Z$ to the nearest skew-symmetric Hamiltonian matrix $\mathrm{skh}(A_\mathcal{H})$ is reported. The results are averaged over 10 runs.}
    \label{tab:hamiltonian_form}
\end{table}

\begin{lemma}
\label{thm:nearest_skh}
Consider a matrix $B \in \mathbb R^{2n\times2n}$ partitioned into $n\times n$ blocks,
$
    B=
    \begin{bmatrix}
        B_{11}&B_{12}\\
        B_{21}&B_{22}
    \end{bmatrix}.
$ 
Then the orthogonal projection of $B$ onto $\mathrm{SkH}(2n)$ is given by
\[
    \mathrm{skh}(B)
    =
    \frac12
    \begin{bmatrix}
        \mathrm{skew}(B_{11}+B_{22})
        &
        -\mathrm{sym}(B_{21}-B_{12})
        \\
        \mathrm{sym}(B_{21}-B_{12})
        &
        \mathrm{skew}(B_{11}+B_{22})
    \end{bmatrix}.
\]
Consequently,
$    \mathrm{skh}(B)
    =
    \argmin_{M\in\mathrm{SkH}(2n)}
    \|B-M\|_{\mathrm F}.
$
\end{lemma}

\begin{proof}
Every $M\in\mathrm{SkH}(2n)$ has the form
\[
    M=
    \begin{bmatrix}
        \Omega&-H\\
        H&\Omega
    \end{bmatrix},
    \qquad
    \Omega^\top=-\Omega,
    \quad
    H^\top=H.
\]
Hence
\[
    \|B-M\|_{\mathrm F}^2
    =
    \|B_{11}-\Omega\|_{\mathrm F}^2
    +\|B_{22}-\Omega\|_{\mathrm F}^2 
    +\|B_{12}+H\|_{\mathrm F}^2
    +\|B_{21}-H\|_{\mathrm F}^2.
\]
The minimization with respect to $\Omega$ and $H$ separates, and the
unique minimizers are
$
    \Omega
    =
    \frac12\mathrm{skew}(B_{11}+B_{22})$,
    $H
    =
    \frac12\mathrm{sym}(B_{21}-B_{12})$,
which gives the stated formula.
\end{proof}

Step 6 requires a Hermitian eigensolver. Stable Hermitian eigensolvers are available, see, e.g., \href{https://www.netlib.org/lapack/explore-html/d8/d1c/group__heev_gabba143a47eee873cbdb00d685fca08a3.html#gabba143a47eee873cbdb00d685fca08a3}{\texttt{heev}}/\href{https://www.netlib.org/lapack/explore-html/d1/d56/group__heevr_ga5b40f20c2f8ddaeba9345489b408f94e.html#ga5b40f20c2f8ddaeba9345489b408f94e}{\texttt{heevr}}/\href{https://www.netlib.org/lapack/explore-html/d8/d30/group__heevd_ga67c8efc78670b5cf75bb7018ec519a14.html#ga67c8efc78670b5cf75bb7018ec519a14}{\texttt{heevd}} from LAPACK. Finally,
Step~7 consists of
multiplications by numerically orthogonal matrices and does not
significantly amplify the errors already present in $Z$ and $U$.
The numerical accuracy of the complete method is investigated in
\Cref{sec:experiments}.

\subsection{Matrices of odd size}\label{sec:odd} 


In this section, we consider a skew-symmetric matrix of odd size: $A\in\mathrm{Skew}(2n+1)$. Note that $A$ cannot have a skew-symmetric \emph{and} orthogonal polar factor because every odd-sized
skew-symmetric matrix is singular. In the singular case, a polar decomposition (in the sense of~\Cref{def:polar}) is not unique and we will instead consider the so called \emph{canonical partial polar decomposition}~\cite[Theorem 8.3]{Higham2008} defined by
\[
 A = P Y, \quad Y = (A^\top A)^{1/2}, \quad P = A Y^\dagger.
\]
Note that $P$ is skew-symmetric and a partial isometry, but it is no longer orthogonal.
QDWH can also be used to compute the canonical partial polar factor
of a rank-deficient matrix, provided that some care is taken in its implementation; see \cite[Secs.~5.4--5.5]{NakatsukasaHigham2013}.

Generically, $\dim\ker(A)=1$, and we focus on this case.\footnote{If $\dim(\ker(A))>1$, a sufficiently small random
skew-symmetric perturbation reduces the nullity to one with probability one.}
 By definition, 
$P^\top P$ is the orthogonal projector onto $\mathrm{range}(A^\top)
=\ker(A)^\perp$. Hence, for a unit vector $z$ such that $Az = 0$, we have 
$-P^2=P^\top P=I_{2n+1}-zz^\top$,
and thus $I_{2n+1}+P^2=zz^\top$.
Consequently, for a random vector $x\in\mathbb R^{2n+1}$,
\[
    z=
    \frac{(I_{2n+1}+P^2)x}
         {\|(I_{2n+1}+P^2)x\|_2}
\]
holds, up to sign, with probability one.
On $\text{range}(z)^\perp$, the partial isometry $P$ is orthogonal. Therefore, the construction of \cref{thm:qfactor} can be
applied on this subspace. Concretely, we sample
$\widehat V\in\mathbb R^{(2n+1)\times n}$ with independent standard
normal entries and set
\[
    V=(I_{2n+1}-zz^\top)\widehat V.
\]
Then define
\[
    \widetilde V
    =\texttt{qf}\bigl((P+\mathrm{i}I_{2n+1})V\bigr),
    \qquad
    Z=\sqrt{2}
    \bigl[\Re(\widetilde V)\mid-\Im(\widetilde V)\bigr].
\]
The matrix $[\,Z\mid z\,]$ is orthogonal, $Az=0$, and
$Z^\top AZ$ is skew-symmetric Hamiltonian. Hence Steps~5--7 of
\cref{alg:hamiltospectral}, applied to $Z^\top AZ$, yield an RSD of
$A$ after appending the zero eigenvalue associated with $z$.

Our implementation of \cref{alg:hamiltospectral}
applies the additional steps explained above to compute an RSD for a skew-symmetric matrix of odd size.%

\section{Numerical experiments}\label{sec:experiments}
This section assesses the accuracy and running time of
\cref{alg:hamiltospectral} and compares it with alternative approaches
for solving skew-symmetric EVPs. We consider the following methods:
\begin{enumerate}
    \item \cref{alg:hamiltospectral}, using QDWH (see \Cref{sec:app_polar_factor}) to compute the polar
    factor unless stated otherwise.
    \item An RSD based on skew-symmetric tridiagonalization followed by
    a bidiagonal SVD. The tridiagonalization is performed with the
    level~3 BLAS implementation available in the Julia package
    \href{https://github.com/JuliaLinearAlgebra/SkewLinearAlgebra.jl}{SkewLinearAlgebra.jl}.
    \item An RSD obtained from the Julia routine \texttt{schur}, which
    calls \texttt{gees} from LAPACK.
    \item An EVD of the Hermitian matrix $\mathrm{i}A$, computed with
    the Julia routine \texttt{eigen}.
    This approach exploits that $\mathrm{i}A$ is Hermitian, but it produces a complex
eigenbasis rather than an RSD of $A$. Recovering an RSD requires additional postprocessing, particularly in the presence of multiple or tightly clustered eigenvalues.
\end{enumerate}

\subsection{Numerical experiments on the accuracy}
We investigate the accuracy of \cref{alg:hamiltospectral} as the
matrix size and conditioning vary. The matrix size ranges from $16$
to $4096$, and the test matrices are sampled from Distributions~1
and~2 of \Cref{sec:app_polar_factor}. The residuals reported in \cref{fig:hamilto_spectral} remain small over
the entire range.

\begin{figure}[ht]
    \centering
    \includegraphics[width=0.48\linewidth]{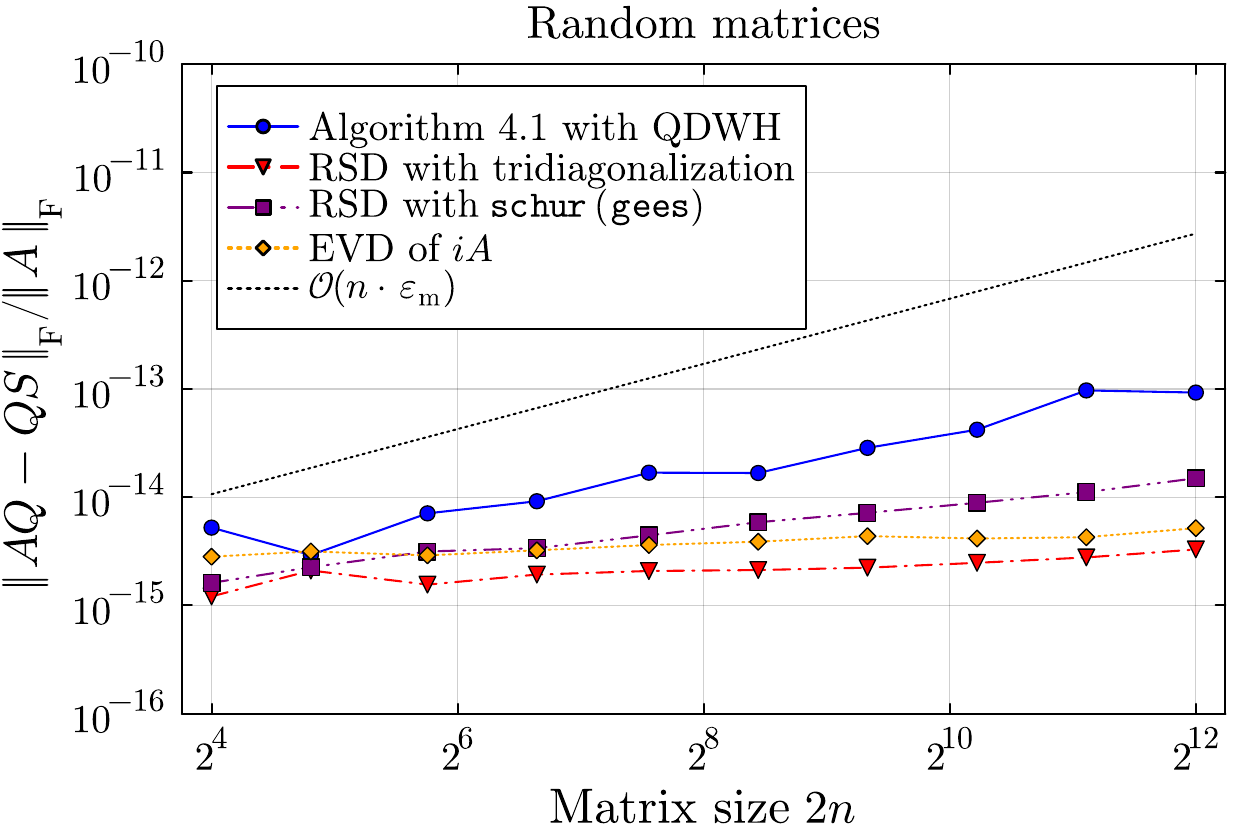}
    \includegraphics[width = 0.48 \linewidth]{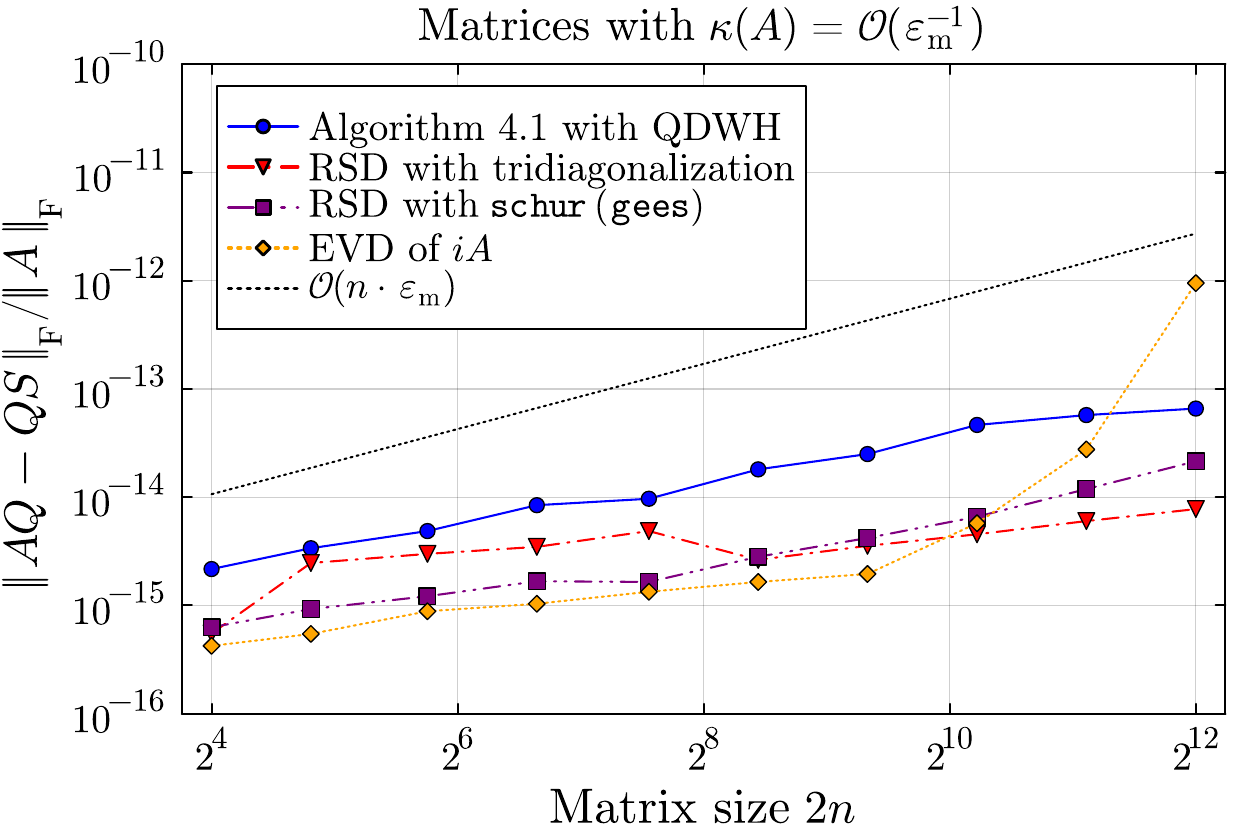}
    \caption{Comparison of the numerical accuracy of \cref{alg:hamiltospectral} with the other algorithms for computing an RSD or an EVD. The test matrices are sampled from the distributions of \Cref{sec:app_polar_factor}. The results are averaged over 10 runs.}
    \label{fig:hamilto_spectral}
\end{figure}    
\subsection{Running time experiments}

We compare the running times of \cref{alg:hamiltospectral} with the
alternative approaches on random skew-symmetric matrices of increasing
size. The results are shown in \cref{fig:time}. All timings were obtained on a processor i7-8750H CPU @ 2.20GHz using Julia 1.12.3 with OpenBLAS. The number of BLAS threads used in each experiment is specified in the corresponding figure or table caption.

\cref{fig:time} shows that \cref{alg:hamiltospectral} competes with the alternative approaches for sufficiently large matrices. However, in this experimental setting and with this implementation, it does not outperform them.
\begin{figure}[ht]
    \centering
    \includegraphics[width=0.48\linewidth]{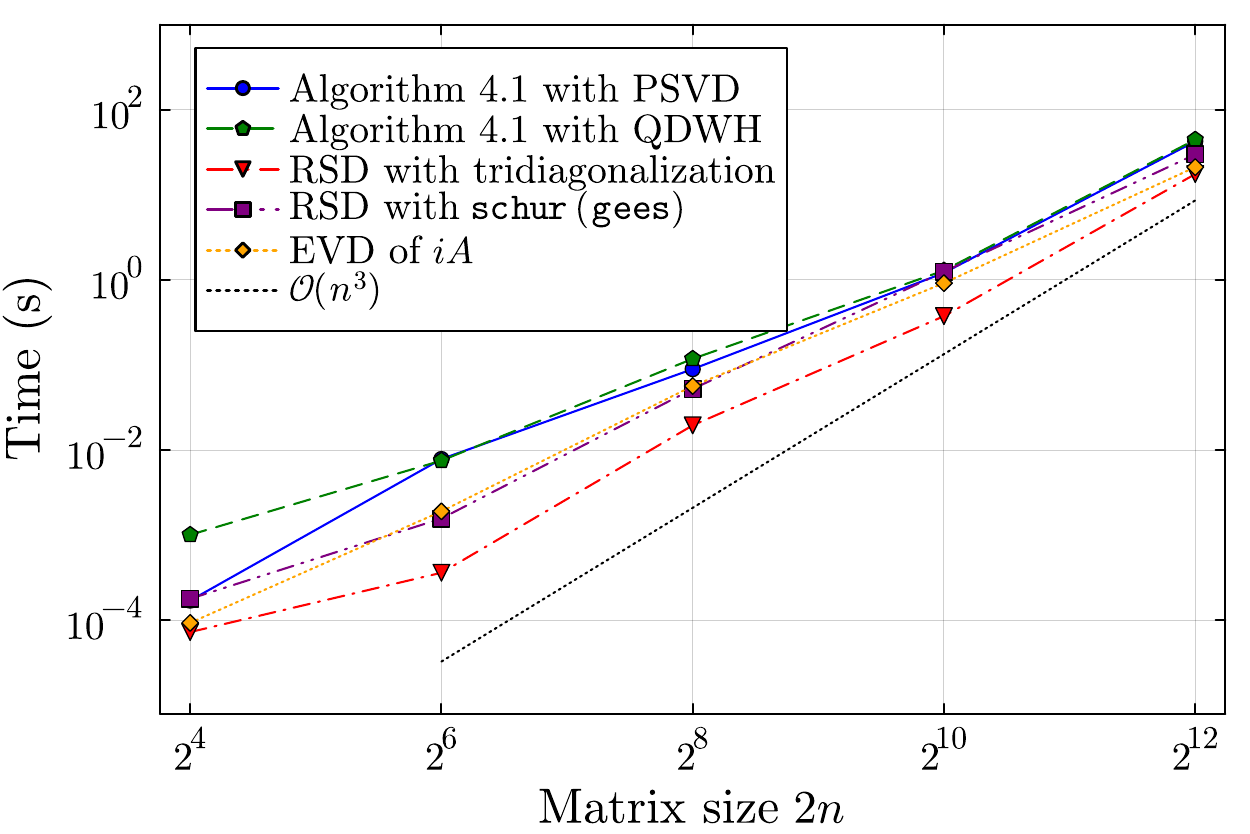}
    \includegraphics[width = 0.48\linewidth]{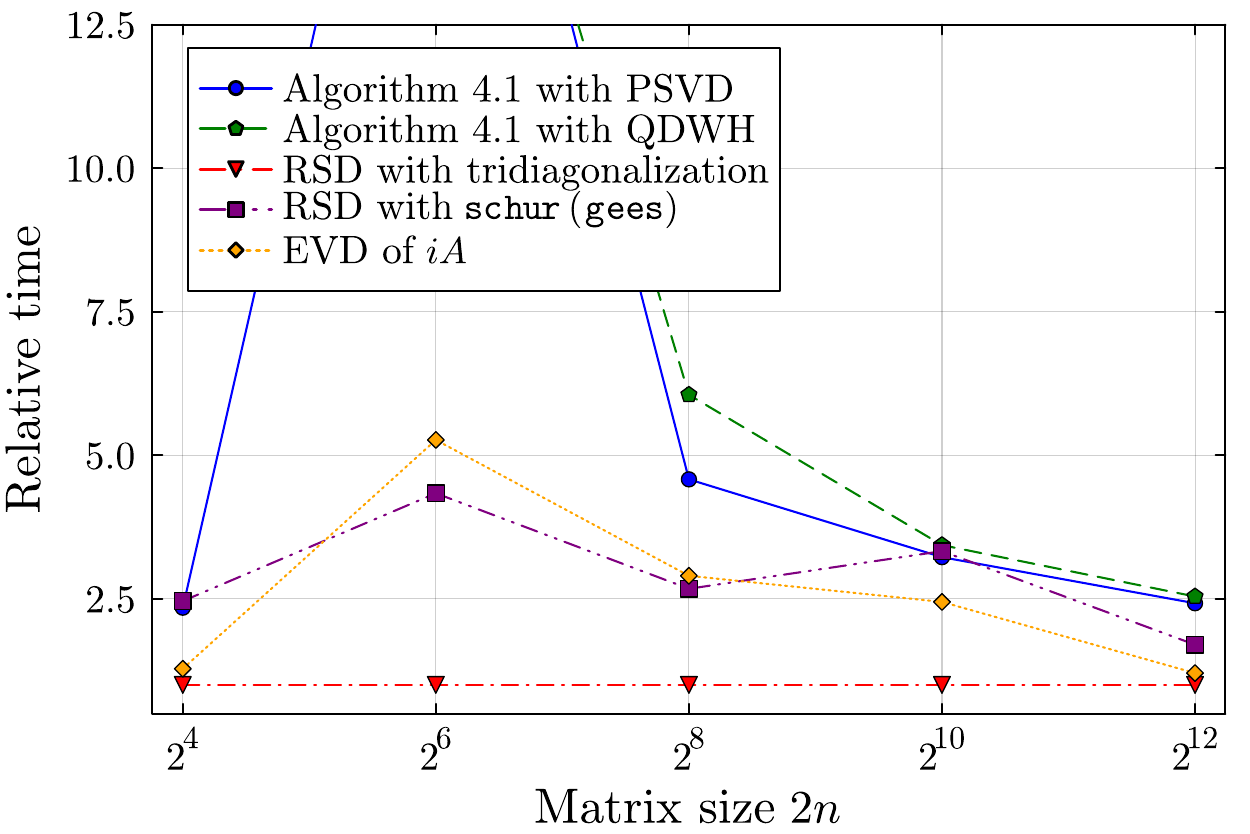}
    \caption{Evolution of the running times of the different algorithms  for $2n\times 2n$ random skew-symmetric matrices with 6 BLAS threads. On the left, the absolute running time in seconds. On the right, the relative running time compared to RSD with tridiagonalization. The results are averaged over 10 runs.}
    \label{fig:time}
\end{figure}

In \cref{fig:pies_time}, we evaluate the proportion of time allocated to each step of \cref{alg:hamiltospectral}. In order to increase readability, the different steps are grouped by ``main operation''. The results are displayed as pie charts. In our setting, the computation of the polar factor accounts for about
$60\%$--$75\%$ of the total running time. Since its main computational
kernels are highly amenable to parallelization~\cite{Ltaief19}, this
breakdown suggests that the overall method could benefit substantially
from a parallel implementation.

\begin{figure}
    \centering
    \includegraphics[width=\linewidth]{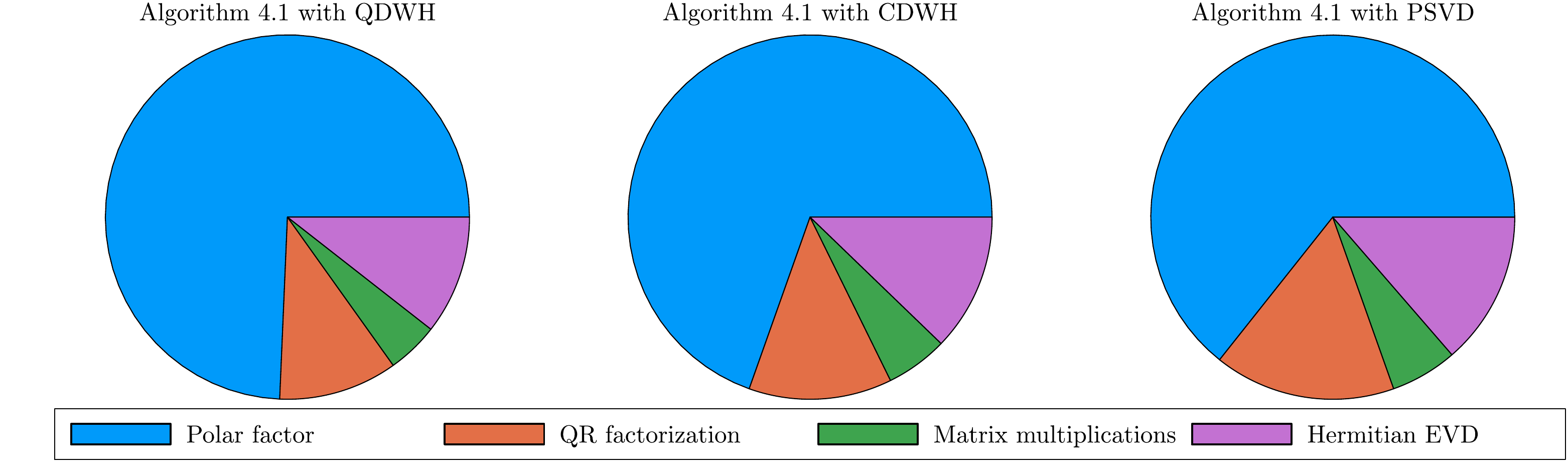}
    
    \vspace{-0.3cm}
    
    \caption{Proportion of time spent on each step of \cref{alg:hamiltospectral} for $2000\times 2000$ random skew-symmetric matrices with $4$ BLAS threads. ``Polar factor'' corresponds to step~1, ``QR factorization'' corresponds to steps~2 to~4. ``Matrix multiplications'' corresponds to steps~5 and~7. ``Hermitian EVD'' corresponds to step~6. Step~1 is implemented either with QDWH, CDWH or PSVD (see \Cref{sec:app_polar_factor}). The results are averaged over 10 runs.}
    \label{fig:pies_time}
\end{figure}

\subsection{Application: simulation of hyperbolic partial differential equations}

We evaluate the performance of \cref{alg:hamiltospectral} on skew-symmetric matrices obtained from an application. Skew-symmetry is an important structure for conservation laws in the simulation of hyperbolic partial differential equations (PDEs), for example when using skew-symmetric formulations~\cite{DucrosEtAl2000}, summation-by-parts schemes~\cite{SvardNordstrom2014,Gassner2013} or geometric numerical integration~\cite{HairerLubichWanner2006,IserlesMuntheKaasNorsettZanna2000}.
Here, we focus on linear semidiscrete systems whose evolution can be
expressed in terms of a matrix exponential; see, e.g.,
\cite{HochbruckOstermann2010}.

Linear hyperbolic PDEs can often be discretized in space as linear
differential equations
\[
    \frac{\mathrm d}{\mathrm dt}u^{(N)}(t)
    +A^{(N)}u^{(N)}(t)=0,
    \qquad
    A^{(N)}\in\mathrm{Skew}(N),
\]
whose solution is
$
    u^{(N)}(t)=\exp(-tA^{(N)})u^{(N)}(0).
$
An RSD of $A^{(N)}$, computed for example with
\cref{alg:hamiltospectral}, therefore allows the semidiscrete solution
to be evaluated directly for arbitrary times.

As an illustration, consider the linear advection equation in
conservative form
\begin{equation}\label{eq:transport}
    \frac{\partial u}{\partial t}
    +\nabla\cdot(cu)=0,
    \qquad
    u(x,0)=f(x),
\end{equation}
on a periodic domain $\mathcal D$. Writing $u=w^2$ yields the
skew-adjoint formulation
\begin{equation}\label{eq:transport_w}
    \frac{\partial w}{\partial t}
    +\frac12\bigl(c\cdot\nabla w+\nabla\cdot(cw)\bigr)=0.
\end{equation}
Indeed, every solution of \eqref{eq:transport_w} gives a solution
$u=w^2$ of \eqref{eq:transport}. In one space dimension, let $D$
denote a periodic centered finite-difference differentiation matrix and
$C=\diag(c(x_1),\ldots,c(x_N))$. Since $D^\top=-D$, the matrix
$
    A^{(N)}=\frac12(CD+DC)
$
corresponding to this discretization of~\eqref{eq:transport} is skew-symmetric.

\Cref{fig:transport1d} compares the resulting RSD-based solution with
Crank--Nicolson and with a reference solution obtained by the method of
characteristics. We use
\[
    w(x,0)=e^{-x^2/2}\cos(\pi x),
    \qquad
    u(x,0)=w(x,0)^2=e^{-x^2}\cos^2(\pi x).
\]
The RSD evaluates the semidiscrete evolution exactly in time and hence
preserves $\|w^{(N)}(t)\|_2$. While Crank--Nicolson is also norm-preserving
and unconditionally stable for this skew-symmetric system, it 
introduces numerical dispersion.

Finally, \cref{tab:transport_performance} reports the cost and
residual of \cref{alg:hamiltospectral} for the matrices
$A^{(N)}$ arising from this discretization. Although these matrices are
sparse, this example illustrates the structure-preserving properties
of the proposed RSD algorithm; a full decomposition can be useful when
the evolution needs to be evaluated for many times or initial
conditions.

\begin{figure}[ht]
    \centering
    \includegraphics[width=\linewidth]{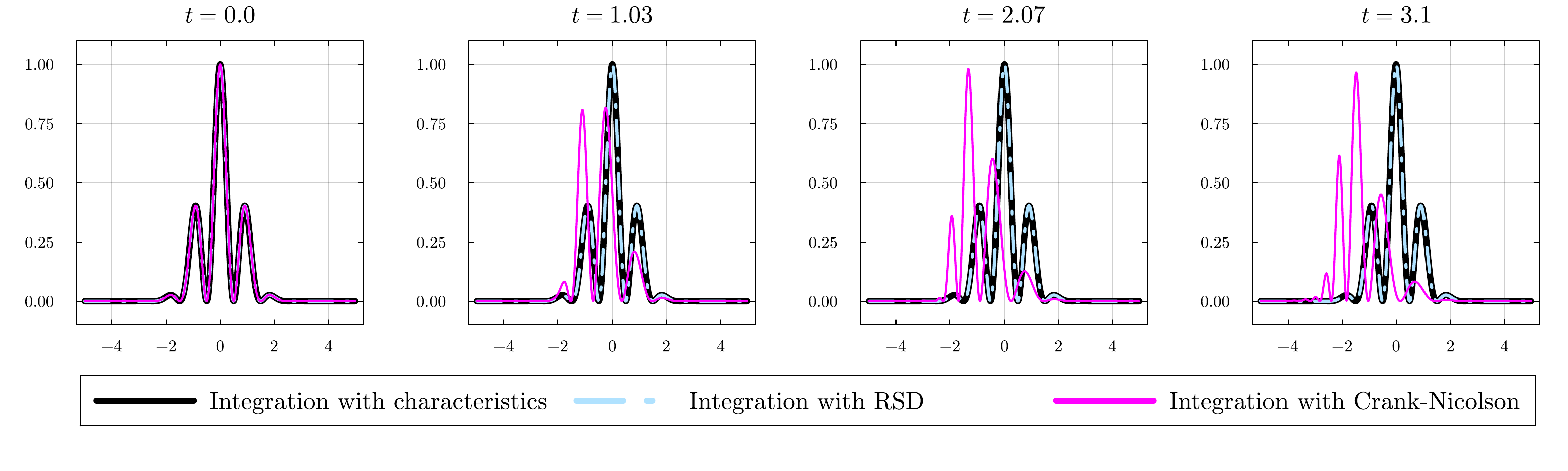}

    \vspace{-0.3cm}
    
\caption{Numerical solutions to~\eqref{eq:transport} obtained from the
skew-adjoint formulation \eqref{eq:transport_w} on the periodic domain
$\mathcal D=[-L,L)$, with $L=5$,
$c(x)=2L+\frac{L}{2}\cos(2\pi x/L)$, and
$w(x,0)=e^{-x^2/2}\cos(\pi x)$. We use $N=500$ grid points and a
sixth-order centered finite-difference discretization. The reference
solution is obtained by the method of characteristics.
    \label{fig:transport1d}}
\end{figure}


\begin{table}[ht]
    \centering
    \begin{tabular}{||c|c|c||}
        \hline
         $N$& $\|A^{(N)}Q - QS\|_\mathrm{F}/\|A^{(N)}\|_\mathrm{F}$& Running time (s)  \\
         \hline
         10&1.92e-15 & 1.78e-4 \\ 
         100& 6.27e-15& 1.66e-2 \\
         1000& 4.54e-14& 1.36\\
         \hline
    \end{tabular}
    
    \caption{Performance of \cref{alg:hamiltospectral} using QDWH with
four BLAS threads for computing an RSD of the discrete operator
$A^{(N)}$.}
    \label{tab:transport_performance}
\end{table}

\section{Halving the size of orthogonal eigenvalue problems}\label{sec:special_orthogonal}

Analogously to the transformation of a skew-symmetric matrix to
skew-symmetric Hamiltonian form, an orthogonal matrix can be transformed
into an orthogonal symplectic matrix. The resulting eigenvalue problem is
equivalent to a unitary eigenvalue problem of half the size.

We first consider the case
\begin{equation}
    X\in\mathrm{SO}(2n),
    \quad 
    \text{$X$ has no eigenvalues $\pm 1$}.
    \label{eq:orthogonal-generic-assumption}
\end{equation}
The spectral condition in
\eqref{eq:orthogonal-generic-assumption} holds on an open dense subset
of full Haar measure\footnote{Indeed, the exceptional set is the zero set on $\mathrm{SO}(2n)$ of
the polynomial $
    X\longmapsto \det(X-I_{2n})\det(X+I_{2n}),
$
and therefore has Haar measure zero.} in $\mathrm{SO}(2n)$.

Under assumption~\eqref{eq:orthogonal-generic-assumption}, there exist
$Q\in\mathrm{O}(2n)$ and
$\Sigma=\operatorname{diag}(\theta_1,\ldots,\theta_n)$, with $0<\theta_j<\pi$,
such that
\begin{equation}
    X
    =
    Q
    \begin{bmatrix}
        C & -S\\
        S & C
    \end{bmatrix}
    Q^\top,
    \qquad
    C\coloneq\cos(\Sigma),
    \quad
    S\coloneq\sin(\Sigma);
    \label{eq:orthogonal_rsd}
\end{equation}
see, e.g., \cite[Corollary 2.5.11 (c)]{Horn_Johnson_2012}.
In particular, $S\succ0$ and $C^2+S^2=I_n$.
In this section only,
$S$ denotes the diagonal matrix of sines rather than the spectral
form~\eqref{eq:spectral_decomposition}.
Note that $\left[\begin{smallmatrix}
        C&-S\\
        S&C
    \end{smallmatrix}\right]$ is in the orthogonal symplectic group $\mathrm{OSp}(2n)$. Therefore, for every $Z = QM$ with $M\in\mathrm{OSp}(2n)$, we have
\begin{equation*}
    Z^\top X Z = M^\top \begin{bmatrix}
        C&-S\\
        S&C
    \end{bmatrix} M=\begin{bmatrix}
        \widetilde{U}_\mathrm{r}&-\widetilde{U}_\mathrm{i}\\
        \widetilde{U}_\mathrm{i}&\widetilde{U}_\mathrm{r}
    \end{bmatrix}\in \mathrm{OSp}(2n).
\end{equation*}
Finding an RSD of the orthogonal symplectic matrix on the right-hand side is equivalent to solving the unitary EVP $\widetilde U = \widetilde{U}_\mathrm{r} + \mathrm{i} \widetilde{U}_\mathrm{i}$ of size $n$. Specialized QR algorithms for unitary matrices are available; see~\cite{Gragg1986, Stewart2006}. More recently, a simple randomized algorithm that solves the unitary EVP through a Hermitian EVP of size $n$ was presented in~\cite{HeKressner2025}.

Instead of starting from~\eqref{eq:orthogonal_rsd}, the required transformation of $X$ to orthogonal symplectic form can also be obtained from the skew-symmetric part of $X$.

\begin{proposition}
\label{prop:orthogonal-to-symplectic}
Let $X$ satisfy~\eqref{eq:orthogonal-generic-assumption}, and set
$A\coloneq\operatorname{skew}(X)$. Then $P\coloneq A(-A^2)^{-1/2} \in\mathrm{Skew}(2n)\cap\mathrm{O}(2n)$ is the unique
orthogonal polar factor of $A$.
If $Z\in\mathrm{O}(2n)$ satisfies
$P=ZJ_{2n}Z^\top$, then
$
    Z^\top XZ\in\mathrm{OSp}(2n).
$
\end{proposition}

\begin{proof}
The assumption that $X$ has no eigenvalues $\pm1$ implies that
$A=\operatorname{skew}(X)$ is nonsingular. Since $A^\top=-A$, we have
$
    A^\top A=-A^2\succ0,
$
and hence $P=A(-A^2)^{-\frac{1}2}$ is the unique orthogonal polar factor of
$A$. Because $A$ commutes with $(-A^2)^{-\frac12}$, it follows that $P$ is
skew-symmetric.

To prove the second part, we first note that
$
    A=(X-X^{\top})/2
$
commutes with $X$ and therefore $P$, being a matrix function of
$A$, also commutes with $X$. Therefore,
\[
    J_{2n}(Z^\top XZ)
    =
    Z^\top PXZ
    =
    Z^\top XPZ
    =
    (Z^\top XZ)J_{2n}.
\]
Since $Z^\top XZ$ is orthogonal, it belongs to
$\mathrm{OSp}(2n)$.
\end{proof}

The matrix $P$ from \Cref{prop:orthogonal-to-symplectic} can alternatively be obtained from the matrix sign function.
Indeed, assumption~\eqref{eq:orthogonal-generic-assumption} implies that
$\mathrm{i}X$ has no eigenvalues on the imaginary axis, so its matrix
sign is well defined and it readily follows that 
\begin{equation}
    P
    =
    -\mathrm{i}\,\operatorname{sign}(\mathrm{i}X)
    =
    X(-X^2)^{-\frac12}. 
    \label{eq:skew_sign}
\end{equation}
This sign-based representation can be evaluated,
for example, by the Halley iteration%
\begin{equation}
    P_{k+1}
    =
    P_k(3I_{2n}-P_k^2)(I_{2n}-3P_k^2)^{-1},
    \qquad
    P_0=X.
    \label{eq:sign_iteration}
\end{equation}
Under assumption~\eqref{eq:orthogonal-generic-assumption}, this
iteration converges to $P$ and preserves the group structure:
$P_k\in\mathrm{SO}(2n)$ for $k=0,1,\ldots$; see~\cite{HMMT04}.
However, the factor $I - 3P_k^2$ is neither symmetric nor orthogonal in general, making it difficult to exploit structure when applying its inverse. In contrast, computing a skopf of $\mathrm{skew}(X)$ with one of the iterative methods from \Cref{sec:app_polar_factor} preserves the skew-symmetry of the iterates but orthogonality is only achieved upon convergence.

Let $P = ZJ_{2n} Z^\top$ and $Z\eqcolon [Z_1\ | \ Z_2]$, $\widetilde{U}_\mathrm{r}\coloneq  Z_1^\top (X Z_1)$ and $\widetilde{U}_\mathrm{i}\coloneq Z_2^\top (X Z_1) $. The next step is solving the unitary EVP of $\widetilde{U}_\mathrm{r} + \mathrm{i}\widetilde{U}_\mathrm{i}= U(C+\mathrm{i}S)U^*$ where $U = U_\mathrm{r} + \mathrm{i} U_\mathrm{i}$. Specialized QR algorithms exploit the compact representation of a
unitary Hessenberg matrix and yield fast and stable iterations; see
\cite{Gragg1986,Stewart2006}. In our experiments, however, we use \texttt{schur} (\texttt{gees} from LAPACK) because a competitive implementation of these algorithms is currently unavailable in Julia.


Finally, we obtain the spectral decomposition

\begin{equation*}
    X = Z \begin{bmatrix}
        U_\mathrm{r}&-U_\mathrm{i}\\
        U_\mathrm{i}&U_\mathrm{r}
    \end{bmatrix}\begin{bmatrix}
        I_n&0\\
        0&\mathrm{sign}(S)
    \end{bmatrix}\begin{bmatrix}
        C&-|S|\\
        |S|&C
    \end{bmatrix}\begin{bmatrix}
        I_n&0\\
        0&\mathrm{sign}(S)
    \end{bmatrix}^\top\begin{bmatrix}
        U_\mathrm{r}&-U_\mathrm{i}\\
        U_\mathrm{i}&U_\mathrm{r}
    \end{bmatrix}^\top Z^\top.
\end{equation*}
A pseudo-code of the method for special orthogonal matrices is given in \cref{alg:symplectospectral}. 

\begin{algorithm}[ht]
    \caption{RSD of special orthogonal matrix}
    \begin{algorithmic}
        \STATE \textbf{Input:} $X\in\mathrm{SO}(2n)$  with no eigenvalue $\{\pm1\}$.
        \STATE \textbf{Output:} $Q\in\mathrm{O}(2n)$ and $C,S\in\mathbb{R}^{n\times n}$ diagonal s.t.~$X = Q (I_2\otimes C +J_2 \otimes S )Q^\top$.
        \STATE \textbf{step 1:} Compute a skew-symmetric orthogonal polar factor $P$ of $\mathrm{skew}(X)$.
        \STATE \textbf{step 2:} Sample $V \in \mathbb R^{2n\times n}$ with independent standard normal entries. 
        \STATE \textbf{step 3:} Compute an orthonormal factor $\widetilde{V} = \texttt{qf}((P + \mathrm{i}I_{2n})V)\in\mathbb{C}^{2n\times n}$.
        \STATE \textbf{step 4:} Define $Z_1 = \sqrt{2} \Re(\widetilde{V})$ and $Z_2 = -\sqrt{2} \Im(\widetilde{V})$.
        \STATE \textbf{step 5:} Compute $\widetilde{U}_\mathrm{r}\coloneq  Z_1^\top (X Z_1)$ and $\widetilde{U}_\mathrm{i}\coloneq Z_2^\top (X Z_1) $.
        \STATE \textbf{step 6:} Solve the unitary EVP of $\widetilde{U}_\mathrm{r} + \mathrm{i}\widetilde{U}_\mathrm{i}= U(C+\mathrm{i}S)U^*$ where $U = U_\mathrm{r} + \mathrm{i} U_\mathrm{i}$.
        \STATE \textbf{step 7:} Let $Q \coloneq  [Z_1U_\mathrm{r} + Z_2U_\mathrm{i}\ |\ (-Z_1 U_\mathrm{i} + Z_2 U_\mathrm{r})\mathrm{sign}(S)]$ and $S\leftarrow |S|$.
        \RETURN $Q,C$ and $S$.
    \end{algorithmic}
    \label{alg:symplectospectral}
\end{algorithm}

In \cref{fig:test_symplectospectral}, it is shown experimentally that the method is stable and competes in terms of speed and accuracy with other algorithms. The test matrices are sampled from a Haar-distribution on~$\mathrm{SO}(2n)$ where $2n$ varies from $8$ to $2048$. The comparison is made with \texttt{schur} (\texttt{gees} from LAPACK) and \texttt{RandDiag} from~\cite{HeKressner2025}.
The residual of \texttt{RandDiag} appears to grow proportionally with
$n^2$, consistently with \cite[Fig.~1]{HeKressner2025}, whereas the
residual of \cref{alg:symplectospectral} exhibits a much milder growth,
approximately proportional to $\sqrt n$ in these experiments.

\begin{figure}[ht]
    \centering
    \includegraphics[width=0.48\linewidth]{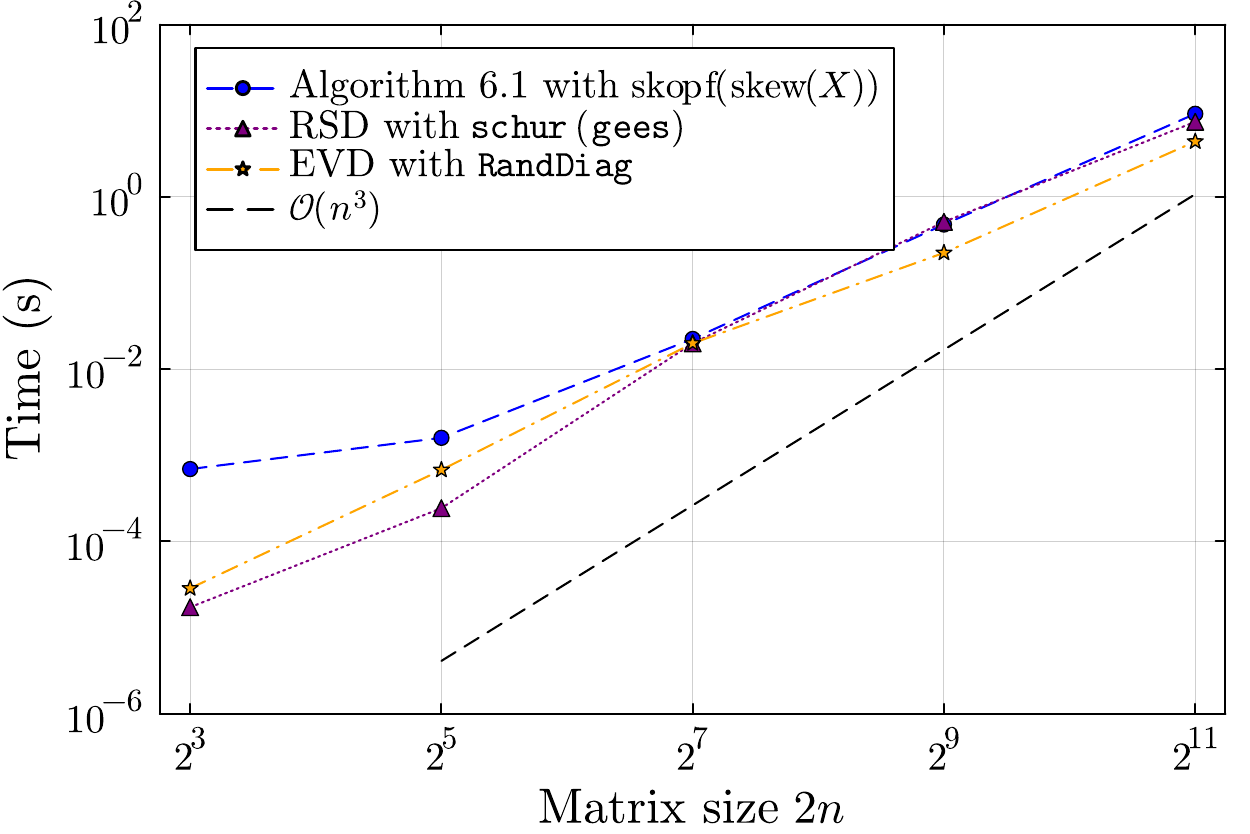}
    \includegraphics[width = 0.48\linewidth]{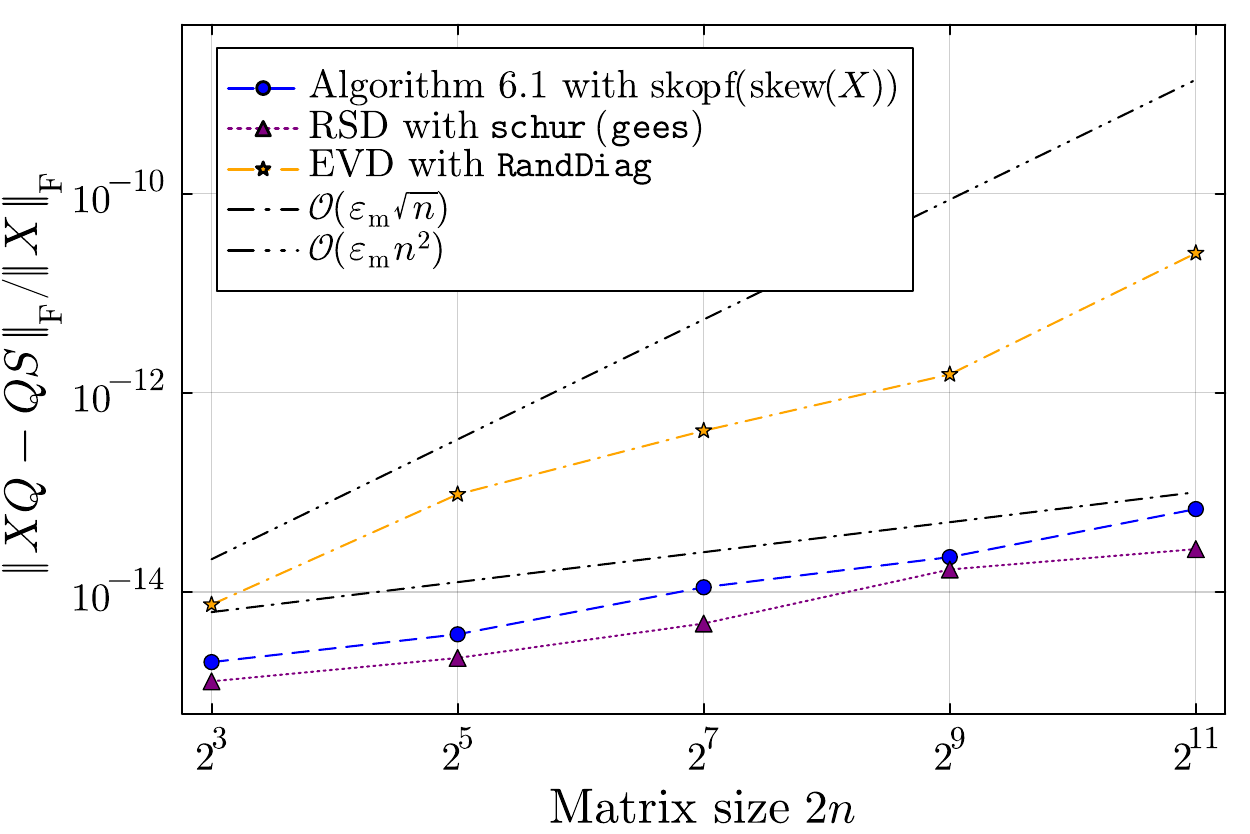}
    \caption{On the left, a comparison of the running time of \cref{alg:symplectospectral} with the Julia routine \texttt{schur} (\texttt{gees} from LAPACK) and the method \texttt{RandDiag} from~\cite{HeKressner2025}. $6$ BLAS threads are used. On the right, a comparison of the numerical accuracy as the size $2n$ of the matrices varies. In the label of the y-axis, $S$ refers to the spectral form~\eqref{eq:spectral_decomposition}. The matrices are sampled from a Haar distribution on $\mathrm{SO}(2n)$. The results are averaged over 10 runs.}
    \label{fig:test_symplectospectral}
\end{figure}

\begin{remark}[Eigenvalues $\pm1$]
The techniques from \Cref{sec:odd} extend to general orthogonal matrices $X$ having
eigenvalues $\pm1$. Indeed,
\[
    \ker(\mathrm{skew}(X))
    =
    \ker(X-I)\oplus\ker(X+I).
\]
Thus, one may compute an orthonormal basis $N$ of
$\ker(\mathrm{skew}(X))$, replace the projection $I-zz^\top$ used in
\Cref{sec:odd} by $I-NN^\top$, and apply the preceding construction
on the orthogonal complement of $\mathrm{range}(N)$. The two real
eigenspaces can be separated by diagonalizing the small symmetric
matrix $N^\top XN$, whose eigenvalues are $\pm1$. 
The basis $N$ can be extracted from the canonical partial polar factor of $\mathrm{skew}(X)$.
\end{remark}

\section{Conclusion}

We have introduced two structure-preserving algorithms for computing real
spectral decompositions of real skew-symmetric and orthogonal matrices.
Their key idea is to exploit a skew-symmetric orthogonal polar factor to
reduce the original eigenvalue problem to a Hermitian or unitary problem
of half the size. Numerical experiments indicate that the resulting
methods are accurate and competitive with existing approaches in terms of
running time. Moreover, the dominant computational steps are all well suited 
for parallelization. Developing and evaluating a scalable parallel implementation is
therefore a natural next step.

\appendix

\section{Polar decomposition of block diagonal matrix} The following lemma is an elementary result about the polar decomposition.

\begin{lemma} \label{lem:polar_plus}
Consider a nonsingular matrix $A_1 \in \mathbb{R}^{n_1 \times n_1}$ and a general matrix 
$A_2 \in \mathbb{R}^{n_2 \times n_2}$. If
$
    A_1 \oplus A_2 = P Y
$
is a polar decomposition, then
\[
    P = P_1 \oplus P_2,
    \qquad
    Y = Y_1 \oplus Y_2,
\]
such that $A_1 = P_1 Y_1$ and $A_2 = P_2 Y_2$ are polar decompositions. 
\end{lemma}

\begin{proof}
By~\cite[Thm.~7.3.1]{Horn_Johnson_2012}, the uniqueness of the positive semidefinite polar factor implies
\[
Y
=
\bigl((A_1 \oplus A_2)^\top(A_1 \oplus A_2)\bigr)^{1/2} \\
=
(A_1^\top A_1)^{1/2} \oplus (A_2^\top A_2)^{1/2}
=: Y_1 \oplus Y_2.
\]
Setting $E_1 = \Big[ {I_{n_1} \atop 0} \Big] \in \mathbb R^{(n_1+n_2)\times n_1}$ and using that $Y_1$ is nonsingular, one obtains
\[
 P E_1 = P Y E_1 Y_1^{-1} = (A_1 \oplus A_2) E_1 Y_1^{-1} = E_1 A_1 Y_1^{-1}.
\]
Hence, $P$ is block upper triangular and, as an orthogonal matrix, this implies that $P$ is actually block diagonal and can thus be written as $P = P_1 \oplus P_2$ for some 
$P_j\in \mathrm{O}(n_j)$. Comparing the diagonal blocks in
$
    A_1\oplus A_2
    =(P_1\oplus P_2)(Y_1\oplus Y_2)
$
gives $A_j=P_jY_j$.
\end{proof}

\section{Computing a skew-symmetric orthogonal polar factor}\label{sec:app_polar_factor}
Most methods considered in this appendix are based on existing matrix iterations for
the orthogonal polar factor taking the form 
\[
    P_{k+1} = P_k r_k(P_k^\top P_k),
\]
with a rational function $r_k$. This iteration preserves the singular
vectors of $P_k$ and aims at driving the nonzero singular values to one.
Structure preservation by such iterations for the polar decomposition (and the closely related matrix sign function) has been studied more generally in matrix groups and structured matrix classes; see, for
example,~\cite{HMMT04,MR2208338}. \cref{prop:skew-projected-polar-iteration} below identifies a class of iterations that preserves skew-symmetric structure and covers all matrix iterations considered in this section. Roundoff error may introduce unstructured perturbations to the iterates. In this situation, the proposition justifies the use of a skew-symmetric projection step, as it does not further enlarge the norm of the error.

\begin{proposition}
\label{prop:skew-projected-polar-iteration}
For real coefficients $a_k,b_k,c_k$, define
\[
    \mathcal F_k(P)
    :=
    P(a_kI+b_kP^\top P)(I+c_kP^\top P)^{-1},    
\]
assuming that the inverse exists. If $P$ is skew-symmetric, then
$\mathcal F_k(P)$ is also skew-symmetric.

Moreover, let $\widehat Y$ be an approximation to
$Y=\mathcal F_k(P)$ and define
$
    \widehat P_+
    :=
    \mathrm{skew}(\widehat Y).
$
Then, for every unitarily invariant norm,
$
    \|\widehat P_+-Y\|
    \leq
    \|\widehat Y-Y\|.
$
\end{proposition}

\begin{proof}
Since $P^\top=-P$, we have $P^\top P=-P^2$, and therefore
\[
    \mathcal F_k(P)
    =
    P(a_kI-b_kP^2)(I-c_kP^2)^{-1}.
\]
The three factors in this expression commute. Moreover,
$(a_kI-b_kP^2)(I-c_kP^2)^{-1}$ is symmetric. It follows that
\[
    \mathcal F_k(P)^\top
    =
    -(a_kI-b_kP^2)(I-c_kP^2)^{-1}P
    =
    -\mathcal F_k(P).
\]

To prove the second claim, note that $Y^\top=-Y$ implies
$
    \mathrm{skew}(\widehat Y)-Y
    =
    \mathrm{skew}(\widehat Y-Y).
$
Consequently, for $E:=\widehat Y-Y$,
\[
    \|\mathrm{skew}(E)\|
    =
    \frac12\|E-E^\top\|
    \leq
    \frac12\bigl(\|E\|+\|E^\top\|\bigr)
    =
    \|E\|.
\]
\end{proof}

\noindent We compare the following five methods for compting a skopf of $A \in \mathrm{Skew}(2n)$:
\begin{paragraph}{\bf PSVD} Given an SVD 
$A=U\Sigma V^\top$, the matrix $UV^\top$ is the unique skopf in exact arithmetic provided that $A$ is invertible.
To account for the effects of finite precision, we return the projected matrix
$
        P=\mathrm{skew}(UV^\top),
$
    which ensures skew-symmetry but may destroy (numerical) orthogonality.
\end{paragraph}
\begin{paragraph}{\bf Newton--Schulz}
    Starting from
$
        P_0={A} / {\|A\|_\mathrm{F}},
$ apply the projected Newton--Schulz iteration
    \[
        P_{k+1}
        =
        \mathrm{skew}\Bigl(
        \frac12P_k(3I-P_k^\top P_k)
        \Bigr),
        \qquad k=0,1,\ldots.
    \]
    This iteration has regained a lot of attention recently in the context of machine learning for its high parallelizability~\cite{GrishinaSmirnovRakhuba2026,jordan2024muon}.
    It converges quadratically once the singular values are
    sufficiently close to one~\cite{BjorckBowie1971}.
    Global convergence deteriorates very quickly for ill-conditioned $A$, as 
    very small singular
    values approach one only very slowly.
\end{paragraph}
\begin{paragraph}{\bf Halley}
    Starting from the same scaled matrix $P_0$, apply the projected
    Halley iteration
    \[
        P_{k+1}
        =
        \mathrm{skew}\left(
        P_k(3I+P_k^\top P_k)
        (I+3P_k^\top P_k)^{-1}
        \right),
        \qquad k=0,1,\ldots.
    \]
    The iteration has cubic local convergence, but may require many
    iterations for severely ill-conditioned matrices~\cite{Gander1990Polar,
    NakatsukasaBaiGygi2010}. The linear system is solved using a
    Cholesky factorization.
\end{paragraph}
\begin{paragraph}{\bf CDWH}
    The Cholesky-based dynamically weighted Halley method~\cite{NakatsukasaBaiGygi2010} 
    \[
        \widehat P_{k+1}
        =
        \frac{b_k}{c_k}P_k
        +
        \left(a_k-\frac{b_k}{c_k}\right)
        P_k(I+c_kP_k^\top P_k)^{-1},
        \qquad
        P_{k+1}=\mathrm{skew}(\widehat P_{k+1}),
    \]
     where the initialization and the parameters $a_k,b_k,c_k$ are chosen
    as in~\cite{NakatsukasaBaiGygi2010}.
    The inverse  is evaluated using a Cholesky
factorization of $I+c_kP_k^\top P_k$, which can induce a loss of accuracy during the initial iterations involving potentially ill-conditioned iterates.
\end{paragraph}
\begin{paragraph}{\bf QDWH}
    The QR-based dynamically weighted Halley iteration evaluates the
    same rational iteration as CDWH through the identity
    \[
        \widehat P_{k+1}
        =
        P_k(a_kI+b_kP_k^\top P_k)
        (I+c_kP_k^\top P_k)^{-1},
        \qquad
        P_{k+1}=\mathrm{skew}(\widehat P_{k+1}),
    \]
    In the QR-based dynamically
    weighted Halley method, abbreviated QDWH, the rational function is
    evaluated without explicitly forming the inverse. Let
    \[
        \begin{bmatrix}
            \sqrt{c_k}P_k\\
            I
        \end{bmatrix}
        =
        \begin{bmatrix}
            Q_{1,k}\\
            Q_{2,k}
        \end{bmatrix}R_k
    \]
    be a thin QR factorization. Then the algebraically equivalent update
    is
    \[
        \widehat P_{k+1}
        =
        \frac{b_k}{c_k}P_k
        +
        \frac{1}{\sqrt{c_k}}
        \left(a_k-\frac{b_k}{c_k}\right)
        Q_{1,k}Q_{2,k}^\top.
    \]
    The convergence of QDWH converges in a small number of iterations even for highly
    ill-conditioned matrices~\cite{NakatsukasaBaiGygi2010}. It also enjoys excellent numerical robustness; in particular,
    backward stability is ensured when a pivoted QR factorization is used~\cite{NakatsukasaHigham2012}. \new{When $c_k$ falls below a prescribed
    threshold, which is set to $10^3$ in our experiments, the method switches to CDWH.
    Thus, potentially ill-conditioned systems are handled by 
    QR factorizations, while the less expensive Cholesky factorization is used in the better-conditioned regime.}
\end{paragraph}

\new{For every method, the convergence threshold is set to $\|P_k^\top P_k - I\|_\mathrm{F} < 10u\sqrt{n}$, where $u$ is the unit-roundoff. The maximum number of iterations is set to $10^2$. For QDWH, the QR factorizations are computed with the routine \texttt{qr} in Julia. It calls \texttt{geqrt} from LAPACK, i.e., a blocked Householder QR algorithm without pivoting.}

In \cref{tab:polar_performance}, the methods are tested on matrices with different condition numbers:
\begin{itemize}
    \item Distribution 1: $A$ is a random skew-symmetric matrix with the upper triangular part composed of independent entries sampled from the standard normal distribution. The condition number of $A$ tents to grow proportionally with the size $2n$ of the matrix.
    \item Distribution 2: $A$ is a numerically singular/ill-conditioned matrix. The spectral basis is sampled from a Haar-distribution and the singular values are exponentially decreasing from $1$ to $2^{-52}$.
\end{itemize}

\color{black}


\cref{tab:polar_performance} shows that QDWH is the most robust method for computing a skopf.  The accuracy of the method CDWH is affected by the potential ill-conditioning of the linear systems involved in the first iterations.
In the context of massively parallel computations, Newton-Schulz iterations~\cite{BjorckBowie1971} can be better since only matrix-matrix products are required. However, for ill-conditioned matrices, an initial iteration of QDWH is required for Newton-Schulz iterations to converge.

\begin{table}[ht]
    \centering
    \begin{tabular}{||c|c|c|c||}
    \hline
         Method & $\|P^\top P - I\|_\mathrm{F} / \sqrt{2n}$& $\|\mathrm{skew}(P^\top A)\|_\mathrm{F}/\|A\|_\mathrm{F}$&Runtime (s) \\
         
         \hline
         \multicolumn{4}{||c||}{Distribution 1: Random matrix ($\kappa(A) \approx 2780$)}\\
         \hline         
PSVD & 4.6e-15 & 1.3e-15 & 3.5\\ 
NS~\cite{BjorckBowie1971} & 5.4e-16 & 4.9e-16 & 6.7\\
Halley~\cite{Gander1990Polar}  & 1.7e-15& 5.7e-16 & 7.2\\
CDWH~\cite{NakatsukasaBaiGygi2010} & 5.7e-16 & \orange{6.5e-13} \nmark & 3.3\\
QDWH~\cite{NakatsukasaBaiGygi2010} & 5.7e-16 & 4.0e-15 & 4.8\\

\hline
        
         \multicolumn{4}{||c||}{Distribution 2: Ill-conditioned matrix ($\kappa(A) \approx \varepsilon_\mathrm{m}^{-1}$)}\\
         \hline
PSVD & \red{0.18} \xmark  & 1.8e-15 & 2.7\\
NS~\cite{BjorckBowie1971} & 5.4e-16  & 2.1e-15 & \red{17.0} \xmark\\
Halley~\cite{Gander1990Polar}  & 5.9e-16  & 1.6e-15 &\red{ 27.0} \xmark\\
CDWH~\cite{NakatsukasaBaiGygi2010} & 7.0e-16  & \red{0.55} \xmark & 6.2\\
 QDWH~\cite{NakatsukasaBaiGygi2010}&5.7e-16 & 5.5e-15 & 7.3\\

\hline
    \end{tabular}
    \caption{Comparison of the performance of the different algorithms for computing a polar factor. Matrices of size $2000\times 2000$ ($n=1000$) are sampled with different condition numbers. Four BLAS threads are allocated. Red crosses (\xmark) indicate results that would either lead to important loss of accuracy of~\cref{alg:hamiltospectral} or lead to important slow-down. QDWH has no red cross and offers the best trade-off between speed and accuracy of the polar decomposition.}
    \label{tab:polar_performance}
\end{table}

\end{document}